\documentclass[11pt]{amsart}

\usepackage[leqno]{amsmath}
\usepackage{amssymb,latexsym,soul,cite,amsthm,color,enumitem,graphicx,mathtools,microtype,accents,comment,tikz}
\usepackage[colorlinks=true,urlcolor=burntsienna,citecolor=burntsienna,linkcolor=burntsienna,linktocpage,pdfpagelabels,bookmarksnumbered,bookmarksopen]{hyperref}
\definecolor{burntsienna}{rgb}{0.91, 0.45, 0.32}
\usepackage[english]{babel}
\usepackage[left=2.5cm,right=2.5cm,top=2.5cm,bottom=2.5cm]{geometry}

\numberwithin{equation}{section}

\newtheorem{theorem}{Theorem}[section]
\theoremstyle{plain}
\newtheorem{lemma}[theorem]{Lemma}
\theoremstyle{plain}
\newtheorem{proposition}[theorem]{Proposition}
\theoremstyle{plain}

\newtheorem{definition}[theorem]{Definition}
\theoremstyle{definition}
\newtheorem{remark}[theorem]{Remark}

\newcommand{\N}{{\mathbb N}}

\newcommand{\R}{{\mathbb R}}

\newcommand{\beq}{\begin{equation}}
\newcommand{\eeq}{\end{equation}}

\newcommand{\bu}{\bar{u}}

\def\Xint#1{\mathchoice
    {\XXint\displaystyle\textstyle{#1}}%
    {\XXint\textstyle\scriptstyle{#1}}%
    {\XXint\scriptstyle\scriptscriptstyle{#1}}%
    {\XXint\scriptscriptstyle\scriptscriptstyle{#1}}%
    \!\int}
\def\XXint#1#2#3{\setbox0=\hbox{$#1{#2#3}{\int}$}
    \vcenter{\hbox{$#2#3$}}\kern-0.5\wd0}
\def\bint{\Xint-}
\def\bint{\Xint-}
\def\dashint{\Xint{\raise4pt\hbox to7pt{\hrulefill}}}
\def\dashiint{\bint\kern-0.15cm\bint}

\DeclarePairedDelimiter{\abs}{\lvert}{\rvert}
\DeclarePairedDelimiter{\norm}{\lVert}{\rVert}

\DeclareMathOperator*{\essosc}{ess\, osc}
\DeclareMathOperator*{\essinf}{ess\, inf}
\DeclareMathOperator*{\esssup}{ess\, sup}

\def\YYint#1#2#3{{\setbox0=\hbox{$#1{#2#3}{\iint}$}
    \vcenter{\hbox{$#2#3$}}\kern-.51\wd0}}
 
\def\Xint#1{\mathchoice
{\XXint\displaystyle\textstyle{#1}}%
{\XXint\textstyle\scriptstyle{#1}}%
{\XXint\scriptstyle\scriptscriptstyle{#1}}%
{\XXint\scriptscriptstyle\scriptscriptstyle{#1}}%
\!\int}
\def\XXint#1#2#3{{\setbox0=\hbox{$#1{#2#3}{\int}$ }
\vcenter{\hbox{$#2#3$ }}\kern-.6\wd0}}

\def\dashint{\Xint-}

\makeatletter
\newcommand{\leqnomode}{\tagsleft@true}
\newcommand{\reqnomode}{\tagsleft@false}
\makeatother

\title[Doubly non linear regularity ]{H\"older regularity for a class of doubly non linear pdes}

\author[F.M.\ Cassanello, E.\ Henriques ]{Filippo Maria Cassanello, Eurica Henriques}

\address[F.M.\ Cassanello]{Dipartimento di Matematica e Informatica
\newline\indent
Universit\`a degli Studi di Cagliari
\newline\indent
Via Ospedale 72, 09124 Cagliari, Italy}
\email{filippom.cassanello@unica.it}

\address[E. \ Henriques]{Departamento de Matemática
\newline\indent
Universidade de Trás-os-Montes e Alto Douro
\newline\indent
Quinta de Prados 5000-103, Vila Real, Portugal}
\email{eurica@utad.pt}


\begin{document}

\begin{abstract}
We prove local H\"older continuity for non negative, locally bounded, local weak solutions to the class of doubly nonlinear parabolic equations
\begin{equation*}
    \partial_t (u^q)- \text{div}(\abs*{Du}^{p-2} Du) = 0, \qquad p>2, \quad 0<q<p-1 .
\end{equation*}
The proof relies on expansion of positivity results combined with the study of an alternative (related to DeGiorgi-type lemmas) and an exponential shift which allows us to  deal with the intrinsic geometry associated to the problem.
\end{abstract}

\maketitle

\section{introduction}
\noindent This paper is devoted to the study of H\"older continuity for the class of doubly non linear parabolic equations 
\begin{equation}
\label{EQ}
    \partial_t (u^q) - \text{div}(\abs*{Du}^{p-2}Du) = 0 \, \quad  \mathrm{in} \ \ \Omega_T , \quad \mathrm{for} \ \ p>2 \ \ \mathrm{and} \ \  \ 0<q<p-1
\end{equation}
where $\Omega_T = \Omega \times (0,T]$ being $\Omega$ a bounded domain of $\R^N$ and $T$ a real positive number.

\vspace{.2cm}

\noindent One of the interests for this class of equations lies on the fact that, based on the values taken by the coefficients $p$ and $q$, it presents, both in time and space, a double degeneracy for $p>2$ and $q>1$, or a double singularity for $1<p<2$ and $0<q<1$; a blend of the two well studied nonlinear parabolic differential equations: the p-Laplace equation and the porous medium equation. In fact, this class includes the diffusion equation, for $p=2$ and  $q=1$; when $p=2$ we get the porous medium equation; while taking $q=1$ we recover the classical $p$-Laplacian. Along side with  the intrinsic mathematical interest, the doubly non linear parabolic equation models, for instance, the turbulent filtration of non-Newtonian fluids through a porous media ( cf. \cite{DT}), dynamic of glaciers and shallow water flows (cf. \cite{BDL} and the references therein).

\vskip 0.3 cm

\noindent In the past years, many contributions have been made to the study of regularity for this class of evolutionary equations, considering the different ranges of the exponents $p$ and $q$. Apart from the classical case $q=1$ (cf. \cite{DiBe}, \cite{DGV}), regularity properties for weak solutions have been proven for the Trudinger's equation, where $q=p-1$, such as Harnack's inequality (cf. \cite{KK}) and H\"older continuity (cf. \cite{KLSU}, \cite{KSU}, \cite{BDL}). In particular, the issue of H\"older regularity in a wider setting has been investigated considering diverse approaches: by making use of energy and logarithmic estimates as well  studying an alternative argument (following quite closely the approach presented in \cite{DiBe} for the p-laplace evolution equation), the H\"older continuity was proved in \cite{Henriques 2020}, for $p>2$ and $0<q<1$; in \cite{Henriques-Laleoglu 2013}, for $1<p<2$ and $q>1$, and  in \cite{PV} for $m\geq 1$ and $p\geq 2$, where the differential equations at hands was taken in the form $ \partial_t (u) - \text{div}(\abs* u^{m-1} |Du|^{p-2}Du) = 0$; by using recent results on the expansion of positivity which rely on an exponential shift (presented in \cite{Henriques 2022}), the H\"older continuity was obtained in \cite{Henriques 2023} within the range $1<p<2$ and $p-1<q<1$; by considering a different approach not relying on any form of exponential shift, the H\"older continuity was proved in \cite{BDLS} for $p>2$ and $0<q<p-1$ and in \cite{LS} for $1<p<2$ and $q>p-1$.
 
\vskip 0.3 cm

\noindent Although the result presented in this work matches the one presented in \cite{BDLS}, it brings a new light to the study of local regularity due to the fact that by using results on the expansion of positivity (presented in \cite{Henriques 2022}), within a intrinsic geometric setting suitable to the exponent range at hands, and studying and alternative argument relying on measure theoretical information provided over cubes (not cylinders) the result follows. In this sense, this work can also be seen as follow up of the previous one \cite{Henriques 2023}. At a certain stage of the proof, our arguments fall into the p-laplace mode and another expansion of positivity (related to this) will be produced. This reasoning, although in a different context, was also used in \cite{Henriques 2023}, which leads us to strongly believe that it maybe a way of finding analogous regularity results completing the route started before. Therefore, the use of the technique explained in this paper could help in balancing the delicate nature of the intrinsic scaling argument, which is needed to compensate the degradation of the equation’s parabolic structure in respect to the heat equation, in the doubly nonlinear framework.

\vskip 0.2 cm

\noindent The main outcome of our work is the following result:
\begin{theorem}
\label{holder}
    Let $u$ be a non negative, locally bounded, local weak solution to \eqref{EQ} in $\Omega_T$. Then, $u$ is locally H\"older continuous in $\Omega_T$. Moreover, there exist constants $b$, $\bar{\gamma} >1$ and $\alpha \in (0,1)$, depending only on the data, such that for every compact set $U_{\tau} \subset \Omega_T$ we have
    \begin{equation}
        \label{eqholder}
        \abs*{u(x_1,t_1)-u(x_2,t_2)} \leq \bar{\gamma} \bar{M} \biggl( \frac{\abs*{x_1-x_2} + \bar{M}^{\frac{p-q-1}{p}} \abs*{t_1-t_2}^{\frac{1}{p}}}{(p,q)-\text{dist}}\biggl)^{\alpha}
    \end{equation}
    for every pair of point $(x_1,t_1)$, $(x_2,t_2) \in U_\tau$, where $\bar{M}= \max\{M,(2b)^{\frac{1}{p-q-1}} ((p,q)-\text{dist})^{\frac{\epsilon_0}{p-q-1}} \}$, $M= ||u||_{L^\infty(\Omega_T)}$ and $(p,q)-\text{dist} =\displaystyle{\inf_{\substack{(y,s) \in \partial_p \Omega \times[0,T]\\ (x,t) \in U_{\tau}}
            } \left\{ \abs*{x-y} + \frac{M^{\frac{p-q-1}{p}}}{(2b)^{\frac{1}{p}}}\abs*{t-s}^{\frac{1}{p}} \right\}}$.
\end{theorem}

\vskip 0.2 cm
\noindent
The statement of H\"older continuity may generate some curiosity, as it's form is quite different from other frameworks, especially concerning the ratio of distances. However, this is classical in the literature for doubly non linear equations, even in simpler case as the $p$-Laplacian (cf. \cite{DiBe}). As we previously underlined, the degradation of the equation's parabolic structure impose us to consider proper intrinsic cylinders and, in turn, to modify the geometry of the problem, giving us a different metric in which we have to work with.

\vskip 0.2 cm

\noindent The paper is organized as follows. In Section 2, we present the definition of weak solution, describe the intrinsic geometry in which we are going to work and present  some known results that will be used along the text. In Section 3, we prove DeGiorgi's type lemmas and the expansions of positivity for the measure theoretical alternatives. In section 4, we combine the previous results and present the reduction of oscillation which leads to the proof of the main theorem. In addition, we present in an appendix the proof of the time mollification considered.

\vskip 0.5 cm

\begin{notation}
    Throughout the paper, for all $U \subset \R^N$ we will denote by $\abs*{U}$ the $N$-dimensional measure of Lesbegue of $U$. By $K_\rho (y)$ we will denote the cube of edge $2\rho$ and center at $y$. Writings like $u\leq v$ in $U$ will mean that $u(x) \leq v(x)$ for a.e. $x \in U$. Most important, $\gamma$ will denote several positive constants, only depending on the data $N$, $p$, $q$ and $\Omega$. Moreover, in many results in section 3 we will define generic cylinders $Q$, $\tilde{Q}$; when not indicated in the statement the notation $Q$, $\tilde{Q}$ will refer to the main cylinders defined in section 2.
\end{notation}

\section{Setting the framework}

\noindent This section is devoted to the presentation of the framework within it we will be working on.

\begin{definition}
    A non negative, measurable function $u \in C(0,T;L_{\text{loc}}^{q+1}(\Omega)) \cap L_{\text{loc}}^p (0,T; H_{\text{loc}}^{1,p}(\Omega))$ is a local weak sub (super) solution to equation \eqref{EQ} in $\Omega_T$ if, for any compact $K \subset \Omega$ and for almost every $0<t_1<t_2<T$, it satisfies
    \begin{equation}
        \label{weak-dfn}
        \int_{K} [u^q \varphi (x,t)]_{t_1}^{t_2} \, dx + \int_{t_1}^{t_2} \int_K (\abs*{Du}^{p-2} Du \cdot D \varphi - u^q \varphi_t ) \, dx dt \leq (\geq) 0
    \end{equation}
    for every non negative test function $\varphi \in H_{\text{loc}}^{1,q+1}(0,T;L^{q+1}(K)) \cap L_{\text{loc}}^p (0,T;H_{0}^{1,p}(K)) $.

    \noindent We will say that $u$ is a local weak solution to \eqref{EQ} if it is both a local super and a sub solution.
\end{definition}

\noindent In what follows  we define the intrinsic geometry of the problem, under the assumption that $u$ is locally bounded (we recall that local boundedness for a class of doubly non linear equations such as \eqref{EQ} has been proven in \cite{Henriques-Laleoglu 2017}).

\vspace{.2cm}

\noindent Let $(x_0,t_0)$ be an interior point of $\Omega_T$ and $R$ be a positive real number so that the cylinder
\begin{equation*}
    \tilde{Q}= K_{8R}(x_0) \times (t_0-R^{p-\epsilon_0},t_0) \subset \Omega_T 
\end{equation*}
for some $\epsilon_0 \in (0,1)$ to be chosen; then define
\begin{equation*}
    \mu_{+}= \esssup_{(x_0,t_0)+\tilde{Q}} u \qquad \mu_{-}= \essinf_{(x_0,t_0)+\tilde{Q}} u \qquad \omega= \mu_{+} - \mu_{-}
\end{equation*}
We will assume, given that solutions of \eqref{EQ} are non negative and it is the interesting case, that $\mu_{-} = 0$ so that $\mu_{+}=\omega$.

\vspace{.2cm}

\noindent The major idea towards H\"older continuity is to construct an \textit{intrinsic} cylinder 
\begin{equation*}
    Q=K_R(x_0) \times (t_0-A \omega^{q+1-p} R^p,t_0), \quad \mathrm{for}\ \ \mathrm{some} \ \ A>1\ , 
\end{equation*}
contained in $\tilde{Q}$, within which we have a pointwise information of the solution $u$, {\it{i.e.}}
\[ u \geq \eta \, \omega, \quad in \ \  Q \ , \]
for some $\eta \in (0,1)$ (depending only on the data). To this end, we present the following alternative. 

\noindent Consider a time level $s \in (t_0-R^{p-\epsilon_0},t_0)$, that will be determined later. Then either
\begin{equation}
    \label{alt1}
    \abs*{K_R(x_0) \cap \{u(x,s) > \frac{\omega}{4} \}} \geq \frac{1}{2}\abs*{K_R}
\end{equation}
or
\begin{equation}
    \label{alt2}
    \abs*{K_R(x_0) \cap \{u(x,s) > \frac{\omega}{4} \}} \leq \frac{1}{2}\abs*{K_R}
\end{equation}
We will show that, for each one of these cases, we find the pointwise information stated above.

\vspace{.3cm}

\noindent In order to study each one of the previous alternatives, we will be needing several known results. To make this text as self contained as possible trying to give to the reader an easy experience, in what follows we present them, starting with a DeGiorgi type lemma and a result on the expansion of positivity (both can be found in \cite{Henriques 2022}).

\begin{lemma}
\label{SHR1}
    Let $u$ be a non negative, local weak solution to \eqref{EQ} in $\Omega_T$. Assume that, for some $(y,s) \in \Omega_T$ and some $R > 0$,
    \begin{equation*}
        \abs*{K_R(y) \cap \{u(x,s) > M \}} \geq \alpha \abs*{K_R}
    \end{equation*}
    for some $M>0$ and some $0<\alpha<1$. Then, there exist $\epsilon$, $\delta \in (0,1)$, depending only on $\alpha$ and on $N$, $p$ and $q$, such that
    \begin{equation*}
         \abs*{K_R(y) \cap \{u(x,t) > \epsilon M \}} \geq \frac{\alpha}{2} \abs*{K_R}
    \end{equation*}
    for all $t \in (s, s+ \delta M^{q+1-p}R^p]$.
\end{lemma}

\begin{proposition}
\label{CLU1}
Let $u$ be a non negative, local weak solution to \eqref{EQ} in $\Omega_T$. Assume that 
\begin{equation*}
         \abs*{K_R(y) \cap \{u(x,t) > \epsilon M \}} \geq \frac{\alpha}{2} \abs*{K_R}
    \end{equation*}
    for all $t \in (s, s+ \delta M^{q+1-p}R^p]$ and $(y,s) \in \Omega_T$. Then there exist $\eta \in (0,1)$, $s_1 \in \N$, depending on $N$, $p$, $q$ and $\alpha$, such that
    \begin{equation*}
        u(\cdot, t) \geq \eta M \quad \text{a.e in $K_{2R}(y)$}
    \end{equation*}
     for all $t \in \biggl( s  + \delta \bigl(\frac{M}{4}\bigl)^{q+1-p}e^{\frac{6^p}{\delta} \bigl(\frac{\epsilon}{2^{s_1}} \bigl)^{p-q-1}} R^p, s + \delta \bigl(\frac{M}{4}\bigl)^{q+1-p}e^{\frac{8^p}{\delta}\bigl(\frac{\epsilon}{2^{s_1}} \bigl)^{p-q-1}} R^p\biggl)$.
\end{proposition}

\vspace{.1cm}

\noindent Next we present three results that can be found in \cite[Chap.I]{DiBe}: an isoperimetric estimate; a Sobolev embedding and a fast geometric convergence lemma, respectively.
\begin{lemma}
\label{isoperi}
    Let $u \in H^{1,1}(K_R(y)) \cap C(K_R(y))$ and $k$ and $l$ real numbers satisfying $k<l$. There exists a positive constant $\gamma$, depending on $N$ and $p$, such that
    \begin{equation*}
        (l-k) \abs*{K_R(y) \cap \{v >l \}} \leq \frac{\gamma R^{N+1}}{\abs*{K_R(y) \cap \{v<k \}}} \int_{K_R(y) \cap \{k<v<l \}} \abs*{Dv} \, dx
    \end{equation*}
\end{lemma}

\begin{lemma}
    \label{sobemb}
    Let $\Omega \subset \R^N$ be a bounded set, $T$  a real number and $Q = \Omega \times (0,T)$. Moreover, let $m \in \N$, $p>1$ be a real number, and $u \in L^{\infty}(0,T; L^m(\Omega)) \cap L^p (0,T; W_0^{1,p}(\Omega))$, then there exist a constant $\gamma$, depending only on $p$, $N$ and $m$, such that
    \begin{equation*}
        \iint_{Q} \abs*{v}^{p\frac{N+m}{N}} \, dx dt \leq \gamma \biggl(\iint_{Q}\abs*{Dv}^p \, dx dt \biggl) \times \biggl( \sup_{0 \leq t \leq T}\int_{\Omega} \abs*{v}^m \, dx  \biggl)^{\frac{p}{N}}
    \end{equation*}
\end{lemma}

\begin{lemma}
    \label{FCL}
    Let $(Y_n)_n$ be a sequence of non negative numbers satisfying
    \begin{equation*}
        Y_{n+1} \leq C b^n Y_n^{1+\alpha}
    \end{equation*}
    where $C$, $b>1$ and $\alpha \in (0,1)$. Then $(Y_n)_n$ converges to 0 as $n \to \infty$, provided
    \begin{equation*}
        Y_0 \leq C^{-1/\alpha} b^{-1/\alpha^2} \, .
    \end{equation*}
\end{lemma}

\section{The study of the two alternatives}

\subsection{The first alternative}
Assume that \eqref{alt1}, then
\begin{equation}
    \label{alt1mod}
    \abs*{K_R(x_0) \cap \{u(x,s) > \sigma_0 \frac{\omega}{4} \}} \geq \frac{1}{2}\abs*{K_R} \qquad \forall \, \sigma_0 \in (0,1] \ .
\end{equation}
We now apply lemma \ref{SHR1} to \eqref{alt1mod}, with $M = \sigma_0 \frac{\omega}{4}$ and $\alpha= \frac{1}{2}$, to obtain
\begin{equation}
\label{alt1bet}
    \abs*{K_R (x_0) \cap \{u(x,t) >\epsilon_1 \sigma_0 \frac{\omega}{4} \}} \geq \frac{\omega}{4} \abs*{K_R} , \quad \forall \, t \in (s, s+ \delta_1 \bigl(\frac{\omega}{4} \sigma_0 \bigl)^{q+1-p} R^p] \ . 
\end{equation}
and then Proposition \ref{CLU1} allows to get
\begin{equation}
    \label{alt1inf}
    u(\cdot,t)> \eta_1 \frac{\omega}{4} \quad \text{a.e. in $K_{2R}(x_0)$}
\end{equation}
for all $t \in \biggl( s  + \delta_1 \bigl(\frac{\omega}{4}\bigl)^{q+1-p}e^{\frac{6^p}{\delta}b_1^{p-q-1}} R^p, s + \delta_1 \bigl(\frac{\omega}{4}\bigl)^{q+1-p}e^{\frac{8^p}{\delta_1}b_1^{p-q-1}} R^p\biggl)$.

\noindent This is the simplest case to be treated. The more evolved one concerns the study of \eqref{alt2}.

\subsection{The second alternative}
If  \eqref{alt1} does not hold, then we are left with \eqref{alt2}. From this we arrive at
\[ \abs*{K_R(x_0) \cap \{\omega - u(x,s) \geq \frac{3 }{4} \omega \}} \geq \frac{1}{2}\abs*{K_R}\]
and then, by setting $\bar{u} = \omega - u$,
\begin{equation}\label{infobaru}
\abs*{K_R(x_0) \cap \{\bar{u}(x,s) \geq \sigma_0 \frac{3 }{4} \omega \}} \geq \frac{1}{2}\abs*{K_R}, \qquad \forall \ \ 0 \leq \sigma_0 \leq 1 \ .
\end{equation}

\noindent Our next purpose is to study the behaviour of $\bar{u}$. Observe that, since $0 \leq u \leq \omega$ in $\tilde{Q}$ we get that $\bar{u} \geq 0$ in $\tilde{Q}$ is a weak solution to the new differential equation 
\begin{equation}
    \label{EQ2}
    \partial_t(\omega - \bar{u})^q + \text{div}(\abs*{D\bar{u}}^{p-2}D\bar{u})=0 \ .
\end{equation}

\noindent We have the following lemma which is a re-phrasal of energy estimates for solutions to equation \eqref{EQ2}
\begin{lemma}
\label{LemmaEE}
    Let $\bu$ be a non negative, local weak solution of \eqref{EQ2} in $\Omega_T$. There exist a positive constant $C$, depending only on the data $N$, $p$, $q$, such that for every cylinder $ Q= K\times(t_1,t_2)$ compactly contained in $ \Omega_T$, every real number $k \in \R$ and every cut-off function $\xi \in C^1_0 (Q)$ vanishing on the lateral boundary of $Q$ one has
    \begin{equation}
    \begin{aligned}
        \label{EE}
        &\sup_{ t_1 \leq \tau \leq  t_2} \int_{K} \tilde{g}_{-} (\bu, k) \xi^p \, dx + \frac{1}{2} \int_{t_1}^{t_2} \int_K \abs*{D(\bu -k)_{-}}^p \xi^p \, dx d\tau \\ & \leq \int_{K \times \{ t_1 \}} \tilde{g}_{-}(\bu,k) \xi^p \, dx + \int_{t_1}^{t_2} \int_{K} \tilde{g}_{-} (\bu,k) \partial_\tau \xi^p \, dx d\tau  + C \int_{t_1}^{t_2} \int_{K} (\bu-k)_{-}^p \abs*{D\xi}^{p} \, dx d\tau
        \end{aligned} 
    \end{equation}
    where $\displaystyle{\tilde{g}_{-}(\bu,k)= q \int_{\bu}^k (\omega-s)^{q-1}(k-s) \, ds}$.
\end{lemma} 
\begin{proof}
     The definition of weak solution for equation \eqref{EQ2} gives us that, for every $\phi \in H_{\text{loc}}^{1,q+1} (0,T;L^{q+1}(K)) \cap L_{\text{loc}}^p(0,T;H_{0}^{1,p}(\Omega)) $,
    \begin{equation*}
        \int_{K} (w -\bu)^{q}(x,\tau) \phi(x,\tau) \, dx\biggl|_{t_1}^{t_2} - \int_{t_1}^{t_2} \int_K \biggl( \abs*{ D\bu}^{p-2} D\bu D\phi+ (\omega - \bu)^{q} \phi_{\tau} \, dx d\tau \biggl) = 0
    \end{equation*}

    \noindent The mollified version becomes (see the appendix for details), for all $\phi \in L_{\text{loc}}^p(0,T;H_{0}^{1,p}(\Omega)) $
    \begin{equation}
    \label{mollified1}
        \begin{aligned}
            & \int_{0}^{T} \int_K \partial_{\tau} [(\omega - \bu)^{q}]_{h}(x,\tau) \phi(x,\tau) \, dx d\tau - \int_{0}^{T} \int_K \biggl[ (D\bu)^{p-1}\biggl]_h D\phi \, dx d\tau \leq \\
            &\int_K (\omega - \bu)^q (x,0) \biggl( \int_0^T \frac{\phi(x,s)e^{-\frac{s}{h}}}{h} \, ds \biggl)\, dx =: P+ E \leq R
        \end{aligned}
    \end{equation}
    where we are considering $\phi = (u-k)_{-}\xi^p \psi_{\epsilon}$ where $\xi$ is as the lemma request and 
    \begin{equation*}
        \psi_{\epsilon}(\tau) = 
        \begin{cases}
            0 & \quad 0 \leq \tau \leq t_1 - \epsilon \\
            1 + \frac{\tau-t_1}{\epsilon} & \quad t_1 - \epsilon \leq \tau \leq t_1 \\
            1 & \quad t_1 \leq \tau \leq  \\
            1 - \frac{\tau-t}{\epsilon} & \quad t \leq \tau \leq t+ \epsilon \\
            0 & \quad t + \epsilon \leq \tau \leq T
        \end{cases}
    \end{equation*}
    for every $t \in (t_1,t_2)$. We now estimate each one of these integrals. Starting with $P$, consider $v_h$ a function so that $(\omega -v_h)^q = [(\omega - \bu)^q]_h$, then
    \begin{equation}
    \label{eq2}
    \begin{aligned}
        P
        = & \int_0^T \int_K \partial_\tau (\omega -v_h)^q(x,\tau) \phi(x,\tau) \, dx d\tau \\
        = & \int_0^T \int_K \partial_\tau (\omega-v_h)^q \xi^p \psi_{\epsilon} (v_h-k)_{-} \, dx d\tau  \\
        & + \int_0^T \int_K \partial_\tau (\omega -v_h)^q ( (u-k)_{-} - (v_h -k)_{-}) \xi^p \psi_{\epsilon} \, dx d\tau \\
        = &  \int_0^T \int_K \partial_\tau (\tilde{g}_{-} (v_h,k)) \xi^p \psi_{\epsilon} \, dx d\tau \\
        &+ \int_0^T \int_K \frac{(\omega - \bu)^q - (\omega - v_h)^q}{h} ((u-k)_{-} - (v_h - k)_{-})) \xi^p \psi_{\epsilon} \, dx d\tau
        \end{aligned}
    \end{equation}
   In the last line, we have use property (i) of Proposition \ref{propmol} and the fact that
    \begin{equation*}
        \begin{aligned}
          \partial_{\tau}(\omega - v_h)^q (v_h-k)_{-} &= q(\omega-v_h)^{q-1} \partial_{\tau}(\omega-v_h)(v_h -k)_{-} \\
          &= \partial_{\tau} q \int_{v_h}^k (\omega-s)^{q-1}(k-s) \, ds = \partial_{\tau} \tilde{g}_{-} (v_h,k)
        \end{aligned}
    \end{equation*}
    We now observe that on the set $\{\bu>v_h \}$,  $(\omega-\bu)^q < (\omega-v_h)^q$ while $0>k-\bu - (k-v_h)$, so the second integral in \eqref{eq2} is positive. In the similar fashion, on $\{\bu<v_h\}$ we get $(\omega-\bu)^q > (\omega-v_h)^q$ while $0<k-\bu - (k-v_h)$. So we arrive at
    \begin{equation}
    \begin{aligned}
        \label{eq3}
       P  \geq & \int_{0}^T \int_K \partial_\tau (\tilde{g}_{-} (v_h,k)) \xi^p \psi_{\epsilon} \, dx d\tau  \\
        = & - \int_0^T \int_K \tilde{g}_{-}(v_h,k) \partial_{\tau} \xi^p \psi_{\epsilon} \, dx d\tau - \int_0^T \int_K \tilde{g}_{-} (v_h,k) \xi^p \partial_{\tau} \psi_{\epsilon} \, dx d\tau \, .
        \end{aligned}
    \end{equation}
    Finally we consider the change of variables $s' = (\omega - s)^q$ and let $h\to 0^{+}$. By Proposition \ref{propmol} (ii), we get $(\omega-v_h)^q =[(\omega-u)^q]_h \to (\omega -u)^q$ in $L^{\frac{q+1}{q}}(\Omega_T)$ (notice that $(\omega-u)^q \in L^{\frac{q+1}{q}}$ since $u \in L^{q+1}$) and 
    \begin{equation*}
        \lim_{h \to 0^{+}}\tilde{g}_{-}(v_h,k) = \lim_{h \to 0^{+}} \int_{(\omega -k)^q}^{(\omega -v_h)^q} (s'^{\frac{1}{q}}+k-\omega) \, ds' = \int_{(w-k)^q}^{(\omega-\bu)^q} (s'^{\frac{1}{q}}+k-\omega)\, ds' = \tilde{g}_{-}(\bu,k)
    \end{equation*}
    so, by taking $h\to 0^{+}$ in \eqref{eq3}, we arrive at
    \begin{eqnarray*}
        & & - \int_0^T \int_K \tilde{g}_{-}(\bu,k) \partial_{\tau} \xi^p \psi_{\epsilon} \, dx d\tau - \int_0^T \int_K \tilde{g}_{-} (\bu,k) \xi^p \partial_{\tau} \psi_{\epsilon} \, dx d\tau \\
        & = & - \int_{t_1}^{t} \int_K \tilde{g}_{-}(\bu,k) \partial_{\tau} \xi^p \, dx d\tau - \frac{1}{\epsilon}\int_{t_1 - \epsilon}^{t_1} \int_K \tilde{g}_{-} (\bu,k) \xi^p  \, dx d\tau - \frac{1}{\epsilon}\int_{t}^{t + \epsilon} \int_K \tilde{g}_{-} (\bu,k) \xi^p  \, dx d\tau      
    \end{eqnarray*}
   since 
    \begin{equation*}
        \partial_{\tau} \psi_{\epsilon} =
        \begin{cases}
            \frac{1}{\epsilon}, \quad t_1-\epsilon \leq \tau \leq t_1 \\
            - \frac{1}{\epsilon}, \quad t \leq \tau \leq t + \epsilon \ ,
        \end{cases}
    \end{equation*}   
    \noindent and finally, by letting $\epsilon \rightarrow 0^+$, have
    \begin{equation}
        - \int_{t_1}^{t} \int_K \tilde{g}_{-}(\bu,k) \partial_{\tau} \xi^p \, dx d\tau \, + \int_K \tilde{g}_{-} (\bu,k) \xi^p  \, dx \biggl|_{t_1}^{t}
    \end{equation}

    \noindent Now we take a look at $R$. Firstly, we observe that
    \begin{equation*}
        \int_0^T \frac{\phi(x,s)e^{-\frac{s}{h}}}{h} \, ds = \int_0^{\frac{T}{h}} \frac{\phi(x,\tau h)e^{-\tau}}{h} \, d\tau \longrightarrow \phi(x,0) \int_0^{\infty} e^{-\tau} \, d\tau = \phi(x,0), \quad h \to 0^+
    \end{equation*}
    and secondly, since  by construction $\psi_{\epsilon}(0) = 0$, we get
    \begin{equation*}
        \int_K (\omega-u)^q(x,0) \phi(x,0) = 0 \, .
    \end{equation*}

    \noindent Finally we address $E$. Observe that, since $Du \in L^p$ and $u \in L_{\text{loc}}^p (0,T; W_{\text{loc}}^{1,p}(\Omega))$, $(D\bu)^{p-1} \in L^{\frac{p}{p-1}}$ and then by making use of Proposition \ref{propmol} (iii)
    \begin{equation*}
        -\int_0^T \int_K [ (D\bu)^{p-1}]_h D\phi \, dx d\tau \longrightarrow  - \int_0^T \int_K (D\bu)^{p-1} D\phi \, dx d\tau \, \mathrm{as} \ \ h \to 0^+ \ .
    \end{equation*}
From this, by taking the space derivatives of the test function, we arrive at
    \begin{equation}
    \label{eq5}
    \begin{aligned}
         &- \int_0^T \int_K (D\bu)^{p-1}  D(\bu-k)_{-} \xi^p \psi_{ \epsilon}\, dx d\tau - p\int_0^T \int_K (D\bu)^{p-1} (\bu-k)_{-} D\xi \, \xi^{p-1} \psi_{ \epsilon}\, dx d\tau \\
        \xrightarrow{\epsilon \to 0^{+}} & \int_{t_1}^{t} \int_K \abs*{D(\bu-k)_{-}}^{p} \xi^p - p (D\bu)^{p-1} (\bu-k)_{-} D\xi \, \xi^{p-1} \, dx d\tau  \ .
        \end{aligned}
    \end{equation}
    Young's inequality gives
    \begin{equation}
        \label{eq6}
        \begin{aligned}
            &-\int_{t_1}^{t} \int_K p(D\bu)^{p-1} (\bu-k)_{-} D\xi \, \xi^{p-1} \, dx d\tau = \int_{t_1}^{t} \int_K p(D(\bu-k)_{-})^{p-1} (\bu-k)_{-} D\xi \, \xi^{p-1} \, dx d\tau \\
            \geq & - \int_{t_1}^{t} \int_K p\abs*{D(\bu-k)_{-}}^{p-1} (\bu-k)_{-} \abs*{D\xi} \, \xi^{p-1} \, dx d\tau \\
            \geq & -\frac{1}{2} \int_{t_1}^{t} \int_K \abs*{D(\bu-k)_{-}}^{p} \xi^p \, dx d\tau - 2^{p-1} p^p \int_{t_1}^{t} \int_K (\bu-k)_{-}^p \abs*{D\xi}^p \, dx d\tau \, 
        \end{aligned}
    \end{equation}
    thereby
    \begin{equation*}
        E \geq \frac{1}{2} \int_{t_1}^{t} \int_K \abs*{D(\bu-k)_{-}}^{p} \xi^p \, dx d\tau - 2^{p-1} p^p \int_{t_1}^{t} \int_K (\bu-k)_{-}^p \abs*{D\xi}^p \, dx d\tau \, .
    \end{equation*}
   Collecting all theses estimates for $P$, $E$ and $R$, we arrive at 
    \begin{equation}
        \label{eq7}
        \begin{aligned}
            & \int_{K \times \{ t\}} \tilde{g}_{-}(\bu,k) \xi^p (x,\tau) \, dx + \frac{1}{2} \int_{t_1}^{t} \int_K \abs*{D(\bu-k)_{-}}^{p} \xi^p \, dx d\tau \\
            \leq & \int_{K \times \{ t_1\}} \tilde{g}_{-}(\bu,k) \xi^p (x,\tau) \, dx + \int_{t_1}^{t_2} \int_K \tilde{g}(\bu,k) \partial_{\tau} \xi^p \, dx d\tau + 2^{p-1}p^p\int_{t_1}^{t_2} \int_K (\bu-k)_{-}^p \abs*{D\xi}^p \, dx d\tau \, .
        \end{aligned}
    \end{equation}
    and, due to the fact that this holds for all $t \in (t_1,t_2)$, just take the supremum on $t$ on the left hand side to obtain the desired estimate. 
\end{proof}

\noindent Next result is the corresponding one for $\bu$ to Lemma \ref{SHR1}: expanding the theoretical information \eqref{infobaru} in time.
    \begin{lemma}
        \label{SHR2}
        Let $\bu$ be a non negative, local weak solution to \eqref{EQ2}. Assume that, for some $(y,s) \in \Omega_T$ and some $R>0$,
        \begin{equation}
        \label{ip1}
            \abs*{K_R(y) \cap \{\bu(x,s) > M \}} \geq \alpha \abs*{K_R(y)}
        \end{equation}
        for some $\omega > M>0$ and some $0<\alpha<1$. Then, there exist $\epsilon$, $\delta \in (0,\frac{1}{2})$, depending only on $\alpha$ and on $N$, $p$ and $q$, such that
        \begin{equation*}
            \abs*{K_R(y) \cap \{\bu(x,t) > \epsilon M \}} \geq \frac{\alpha}{2} \abs*{K_R(y)}
        \end{equation*}
        for all $t \in (s,s+ \delta M^{2-p} \omega^{q-1} R^p] $.
    \end{lemma}
    \begin{proof}
        Consider, without loss of generality $(y,s)=(0,0)$, and take  $M=a \omega$ for some $a \in (0,1)$. Writting the energy estimates \eqref{EE} over the cylinder $Q= K_R \times (0 , \delta M^{2-p} \omega^{q-1} R^p]$ with $k = M$ and $\xi \in C_0^1(K_R)$, $0 \leq \xi(x) \leq 1$,  $\xi \equiv 1$ in $K_{(1-\sigma)R}$, $\abs*{D\xi} \leq \frac{1}{\sigma R}$ (where $\sigma$, $\delta \in (0,1)$ are to be chosen), we obtain
\begin{equation}
    \label{eq8}
        \begin{aligned}
            & \sup_{(0, \delta M^{2-p} \omega^{q-1} R^p]} \int_{K_R} \tilde{g}_{-}(\bu,M) \xi^p (x,\tau) \, dx + \frac{1}{2} \iint_Q \abs*{D(\bu-M)_{-}}^{p} \xi^p \, dx d\tau \\
            \leq & \int_{K_R \times \{ 0\}} \tilde{g}_{-}(\bu,M) \xi^p (x,\tau) \, dx +  C \iint_Q (\bu-M)_{-}^p \abs*{D\xi}^p \, dx d\tau \\
            \leq & \int_{K_R \times \{ 0\}} \tilde{g}_{-}(\bu,M) \xi^p (x,\tau) \, dx + C \frac{M^2 \delta \omega^{q-1}}{\sigma^p} \abs*{K_R} \, .
        \end{aligned}
\end{equation}
Observe that, for every $t \in (0, \delta M^{2-p} \omega^{q-1} R^p]$,
\begin{equation*}
    \begin{aligned}
        & \sup_{(0, \delta M^{2-p} \omega^{q-1} R^p]} \int_{K_R} \tilde{g}_{-}(\bu,M) \xi^p (x,\tau) \, dx + \frac{1}{2} \iint_Q \abs*{D(\bu-M)_{-}}^{p} \xi^p \, dx d\tau  \\
        \geq & \int_{K_{(1-\sigma)R} \times \{ t\}} \tilde{g}_{-}(\bu,M) \, dx = \int_{K_{(1-\sigma)R} \times \{ t\}} \biggl( q \int_{\bu}^M (\omega-s)^{q-1}(M-s) \, ds  \biggl) \chi_{\{ \bu < M\}}\, dx \\
        \geq & \int_{K_{(1-\sigma)R} \times \{ t\}} \biggl( q \int_{\bu}^M (\omega-s)^{q-1}(M-s) \, ds  \biggl) \chi_{\{ \bu < \epsilon M \}}\, dx \\
        \geq & \int_{K_{(1-\sigma)R} \times \{ t\}} \biggl( q \int_{\epsilon M}^M (\omega-s)^{q-1}(M-s) \, ds  \biggl) \chi_{\{ \bu < \epsilon M \}}\, dx \\
        = &  \abs*{K_{(1-\sigma)R} \cap \{u(x,t) < \epsilon M \}} \times \biggl( q \int_{\epsilon M}^M (\omega-s)^{q-1}(M-s) \, ds  \biggl) \, 
    \end{aligned}
\end{equation*}
then, for all $t \in (0, \delta M^{2-p} \omega^{q-1} R^p]$
    \begin{equation}
        \label{eq9}
        \begin{aligned}
             \abs*{K_{(1-\sigma)R} \cap \{u(x,t) < \epsilon M \}}  \leq & \frac{\int_{K_R \times \{ 0\}} \tilde{g}_{-}(\bu,M) \xi^p (x,\tau) \, dx}{q \int_{\epsilon M}^M (\omega-s)^{q-1}(M-s) \, ds} \\
             +& C \frac{M^2 \delta_2 \omega^{q-1}}{\sigma^p} \frac{1}{q \int_{\epsilon M}^M (\omega-s)^{q-1}(M-s) \, ds}  \abs*{K_R} \\
             \leq &\biggl( 1 + \frac{\int_{0}^{\epsilon M} (\omega-s)^{q-1}(M-s) \, ds}{ \int_{\epsilon M}^M (\omega-s)^{q-1}(M-s) \, ds} \biggl) (1-\alpha)\abs*{K_R} \\
             + &C \frac{M^2 \delta \omega^{q-1}}{\sigma^p} \frac{1}{q \int_{\epsilon M}^M (\omega-s)^{q-1}(M-s) \, ds}  \abs*{K_R} \, ,
        \end{aligned}
    \end{equation}
in the last inequality we have used \eqref{ip1} and the fact that 
\begin{equation*}
    \abs*{K_R \cap \{u(x,0) < M\}} = \abs*{K_R} - \abs*{K_R \cap \{ u(x,0) \geq M \}} \leq (1-\alpha)\abs*{K_R} \, .
\end{equation*}

\noindent In order to bound from above the right hand side of \eqref{eq9}, we bound from below the denominator
    \begin{equation*}
        \begin{aligned}
            q \int_{\epsilon M}^M (\omega-s)^{q-1} (M-s) \, ds \geq &
            \begin{cases}
                q(\omega -M)^{q-1} \frac{(M-\epsilon M)^2}{2} & \quad, q>1 \\
                q(\omega -\epsilon M)^{q-1} \frac{(M-\epsilon M)^2}{2} & \quad, 0<q<1
            \end{cases}\\
            \geq & 
             \begin{cases}
                q \, \omega^{q-1} \, (1-a)^{q-1} \frac{M^2}{2^3} & \quad, q>1 \\
               q (\omega -\epsilon M)^{q-1} \frac{M^2}{2^3} & \quad, 0<q<1 \, ,
            \end{cases}\\ 
        \end{aligned}
    \end{equation*}
    and from above the numerator
    \begin{equation*}
        \begin{aligned}
            q \int_{0}^{\epsilon M} (\omega-s)^{q-1} (M-s) \, ds \leq &
            \begin{cases}
                q\, \omega^{q-1} \frac{M^2 - (M-\epsilon M)^2}{2} & \quad, q>1 \\
                q(\omega -\epsilon M)^{q-1} \frac{M^2 -(M-\epsilon M)^2}{2} & \quad, 0<q<1
            \end{cases}\\
            \leq & 
             \begin{cases}
                q\,\omega^{q-1} \epsilon M^2 & \quad, q>1 \\
               q (\omega -\epsilon M)^{q-1} \epsilon M^2 & \quad, 0<q<1 \, ,
            \end{cases}\\
        \end{aligned}
    \end{equation*}
   to arrive at
   \begin{enumerate}
       \item for $q>1$, 
       \begin{equation}
    \label{eq10}
    \begin{aligned}
        \abs*{K_R \cap \{u(x,t) < \epsilon M \}} \leq & \biggl( \biggl( 1 +\frac{\epsilon 2^3  }{(1-a)^{q-1}} \biggl)(1-\alpha) + C\frac{  \delta 2^3}{q \sigma^p (1-a)^{q-1}} + \sigma^N \biggl)\abs*{K_R} \\
        \leq &\biggl(1- \frac{\alpha}{2} \biggl) \abs*{K_R} \, ,
    \end{aligned}   
    \end{equation}
    if we choose $\sigma^N \leq \frac{\alpha}{8}$, $\delta \leq \frac{\alpha}{8}(\frac{(1-a)^{q-1}\sigma^p}{2^3C})$ and $\epsilon \leq \frac{\alpha (1-a)^{q-1}}{2^5 (1-\alpha)}$.

    \item for $0<q<1$ 
    \begin{equation}
        \label{eq11}
        \begin{aligned}
              \abs*{K_R \cap \{u(x,t) < \epsilon M \}} \leq & \biggl( \bigl( 1 +\epsilon 2^3   \bigl)(1-\alpha) + C\frac{  \delta 2^3}{q \sigma^p (1-\epsilon a)^{q-1}} + \sigma^N \biggl)\abs*{K_R} \\
        \leq &\biggl(1- \frac{\alpha}{2} \biggl) \abs*{K_R}  \, ,
        \end{aligned}
    \end{equation}
    if we choose $\sigma^N \leq \frac{\alpha}{8}$, $\epsilon \leq \frac{\alpha}{2^5 (1-\alpha)}$ and $\delta \leq \frac{\alpha}{8}(\frac{(1-\epsilon a)^{q-1}\sigma^p}{2^3C})$.
       
   \end{enumerate}
\noindent Therefore, choosing the smallest of the two possibilities gives us the proof.
    \end{proof}

\vspace{.2cm}

\noindent Lemma \ref{SHR2} provides us with the first step towards the regularity result: for $(y,s)=(x_0,s)$, $M = \sigma_{0}\frac{3}{4}\omega$, $\alpha= \frac{1}{2}$, we can assure the existence of small positive numbers $\epsilon_2$, $\delta_2 \in (0,\frac{1}{2})$, depending only on $p$, $q$, $N$, such that 
    \begin{equation}
        \label{eq12}
        \abs*{K_R (x_0) \cap \{\bu(x,t) > \epsilon_2 \sigma_0 \frac{3}{4}\omega\}} \geq \frac{1}{4} \abs*{K_R(x_0)}
    \end{equation}
    for all $t \in (s,s+\delta_2 (\sigma_0 \frac{3}{4})^{2-p} \omega^{q+1-p} R^p]$. From here we can go to step 2: use the  measure theoretical information \eqref{eq12} to get a pointwise information, an analog to what we have achieved in Proposition \ref{CLU1}.

    \begin{proposition}
        \label{CLU2}
        Let $\bu$ be a nonnegative, local weak solution to \eqref{EQ2} in $\Omega_T$. Assume that it holds
        \begin{equation*}
            \abs*{K_R (x_0) \cap \{\bu(x,t) > \epsilon_2 \sigma_0 \frac{3}{4}\omega\}} \geq \frac{1}{4} \abs*{K_R(x_0)}
        \end{equation*}
        for some $\epsilon_2$, $\delta_2 \in (0,\frac{1}{2})$, $\sigma_0 \in (0,1)$ and for all $t \in (s,s+\delta_2 (\sigma_0 \frac{3}{4})^{2-p} \omega^{q+1-p} R^p]$. Then, there exist $\eta_2 \in (0,1)$, $s_2 \in \N$, depending on $N$, $p$, $q$, such that
        \begin{equation*}
            u(x, t) \geq \eta_2 \frac{3}{4}\omega \quad \text{a.e.} \ x \in  K_{2R} (x_0)
        \end{equation*}
        for all $t \in \biggl(s + \delta_2\bigl(\frac{3}{4} \bigl)^{2-p} \omega^{1+q-p}  e^{ (b_2)^{2-p} \frac{6^p}{\delta_2}}R^p , s + \delta_2 \bigl(\frac{3}{4} \bigl)^{2-p} \omega^{1+q-p}  e^{ (b_2 )^{2-p}\frac{8^p}{\delta_2}}R^p\biggl) $, being $b_2= \epsilon_2 2^{-s_2}$.
    \end{proposition}

    \begin{proof}
    \noindent To simplify notation and via a translation argument, take $x_0=0$. Consider $\sigma_r = e^{-\frac{r}{p-2}}\leq 1$, for all $r \geq 0$. From the hypothesis  we then get
    \begin{equation*}
        \abs*{K_R \cap \{\bu (x, s + \delta_2 e^r \biggl( \frac{3}{4} \omega \biggl)^{2-p} \omega^{q-1} R^p ) \geq \epsilon_2 \frac{3}{4}\omega e^{-\frac{r}{p-2}}\}} \geq \frac{1}{4} \abs*{K_R} \, 
    \end{equation*}
and then
    \begin{equation}
        \label{eq13a}
        \abs*{K_R\cap \{v(x,r) \geq k_0 \}} \geq \frac{1}{4} \abs*{K_R}
    \end{equation}
for the level $k_0 = \epsilon_2 (\delta_2 R^p)^{\frac{1}{p-2}}$ and the new time variable and function 
    \begin{equation*}
        e^r = \frac{t-s}{\delta_2 R^p} \omega^{1-q} \biggl( \frac{3}{4}\omega \bigg)^{p-2} \ , \quad  v(x,r) = \frac{4 \bu (x,t)}{3 \omega} e^{\frac{r}{p-2}} (\delta_2 R^p)^{\frac{1}{p-2}} \, ,
    \end{equation*}
 where $v$ satisfies 
    \begin{equation} \label{eq13}
        -q \partial_r v \biggl(1- \frac{3}{4} e^{-\frac{r}{p-2}} \bigl(\delta_2 R^p \bigl)^{- \frac{1}{p-2}}v \biggl)^{q-1} + \ \text{div}(\abs*{Dv}^{p-2}Dv ) \leqslant 0 \, .
    \end{equation}
 In what follows we consider the pair of cylinders
    \[\tilde{Q}= K_{4R}  \times (\theta(4R)^p, \theta(8R)^p) \subset Q= K_{8R} \times (\theta R^p, \theta (8R)^p)\]
    the sequence of levels 
    \[k_j = \frac{k_0}{2^j} \qquad j=0,\dots,s_2-1 \quad \text{for some $s_2 \in \N$} \]
    and the cutoff function $\xi \in C_0^1 (Q)$ satisfying
\[0 \leq \xi \leq 1, \quad \xi\equiv 1 \quad \text{in $\tilde{Q}$}, \quad \xi= 0 \quad\text{in $\R^{N+1} \backslash Q$} \ , \quad  \abs*{D\xi} \leq \frac{1}{4R} \quad  \text{and} \quad \abs*{\xi_r}\leq \frac{1}{\theta (3R)^p} \] where $\theta \in \R$ is to be chosen.

\noindent We now test \eqref{eq13} with the $\phi = (v-k_j)_{-} \xi^p$ and proceed formally (an accurate way of proceeding must pass through a mollified version of the equation as we did in recovering the energies estimates): 
    \begin{equation}
    \label{eq14}
    \begin{aligned}
        0 \geq-&q \int_{\theta R^p}^{\theta (8R)^p} \int_{K_{8R}} \biggl(1- \frac{3}{4}e^{\frac{-r}{p-2}} (\delta_2R^p)^{-\frac{1}{p-2}}v \biggl)^{q-1} \partial_r v (v-k_j)_{-} \xi^p \, dx dr \\
        - &\int_{\theta R^p}^{\theta (8R)^p} \int_{K_{8R}} \abs*{Dv}^{p-2} Dv \cdot D((v-k_j)_{-})\xi^p) \, dx dr 
        \end{aligned}
    \end{equation}
  To estimate the integral evolving the time derivative, we define 
    \begin{equation*}
        \mathcal{I}(v,k_j)= q \int_{v}^{k_j} \biggl(1 - \frac{3}{4} e^{-\frac{r}{p-2}}(\delta_2 R^p)^{-\frac{1}{p-2}}s \biggl)^{q-1} (k_j - s) \, ds
    \end{equation*}
 and observe that 
    \begin{equation*}
        \begin{aligned}
       & \partial_r \mathcal{I}(v,k_j) = -q \biggl( 1-\frac{3}{4} e^{-\frac{-r}{p-2}} (\delta_2 R^p)^{-\frac{1}{p-2}}v \biggl)^{q-1}(v-k_j)_{-}\partial_r v \\
        + &q\int_v^{k_j} (q-1)\biggl(1-\frac{3}{4} e^{-\frac{-r}{p-2}} (\delta_2 R^p)^{-\frac{1}{p-2}}s  \biggl)^{q-2}(k_j-s) \biggl[ \frac{3}{4} \biggl(\frac{1}{p-2} \biggl)e^{-\frac{r}{p-2}} (\delta_2 R^p)s \biggl] \, ds \ .
        \end{aligned}
    \end{equation*}
   Plugging this into \eqref{eq14} and computing the gradient we arrive at
    \begin{equation}
        \label{eq15}
        \begin{aligned}
            0 \geq &  \int_{\theta R^p}^{\theta (8R)^p} \int_{K_{8R}} \partial_r \mathcal{I}(v,k_j) \xi^p \, dx dr \\
            -& \frac{3}{4}(\delta_2 R^p) \biggl( \frac{q-1}{p-2} \biggl) q \int_{\theta R^p}^{\theta (8R)^p} \int_{K_{8R}} e^{-\frac{r}{p-2}} \int_v^{k_j} \biggl( 1 - \frac{3}{4} e^{-\frac{r}{p-2}} (\delta_2 R^p)^{-\frac{1}{p-2}} s \biggl)^{q-2}s (k_j -s) \, ds \, dx dr \\
        + &\int_{\theta R^p}^{\theta (8R)^p} \int_{K_{8R}} \abs*{D(v-k_j)_{-}}^{p}  \xi^p \, dx dr + p\int_{\theta R^p}^{\theta (8R)^p} \int_{K_{8R}} (D(v-k_j)_{-})^{p-1} D\xi \, \xi^{p-1} (v-k_j)_{-} \, dx dr \\
        =:& I_1 -I_2 +I_3 + I_4 \, .
        \end{aligned}
    \end{equation}
    Using Young's inequality on $I_4$ 
    \begin{equation*}
        \begin{aligned}
           I_4 =  & \ p\int_{\theta R^p}^{\theta (8R)^p} \int_{K_{8R}} (D(v-k_j)_{-})^{p-1} D\xi \, \xi^{p-1} (v-k_j)_{-} \, dx dr \\
            \geq & - p\int_{\theta R^p}^{\theta (8R)^p} \int_{K_{8R}} \abs*{D(v-k_j)_{-}}^{p-1} \abs*{D\xi} \, \xi^{p-1} (v-k_j)_{-} \, dx dr \\
            \geq & - \frac{1}{2}\int_{\theta R^p}^{\theta (8R)^p} \int_{K_{8R}} \abs*{D(v-k_j)_{-}}^{p}\, \xi^{p}  \, dx dr - p^p 2^{p-1} \int_{\theta R^p}^{\theta (8R)^p} \int_{K_{8R}} \abs*{D \xi}^p (v-k_j)_{-}^p \, dx dr \, ,
        \end{aligned}
    \end{equation*}
    hence 
    \begin{equation*}
        0 \geq I_1 - I_2 +  \frac{1}{2}I_3 - p^p 2^{p-1} \int_{\theta R^p}^{\theta (8R)^p} \int_{K_{8R}} (v-k_j)_{-}^p \, D\xi^p \, dx dr \, ,
    \end{equation*}
    so that 
    \begin{equation}
        \label{eq16}
        \begin{aligned}
            \frac{1}{2} \int_{\theta R^p}^{\theta (8R)^p} \int_{K_{8R}} \abs*{D(v-k_j)_{-}}^{p}  \xi^p \, dx dr \leq I_2 - \int_{K_{8R}} \mathcal{I}(v,k_j) \xi^p \, dx \biggl|_{\theta R^p}^{\theta (8R)^p} \\
            + \int_{\theta R^p}^{\theta (8R)^p} \int_{K_{8R}}p^p 2^{p-1} (v-k_j)_{-}^p \, \abs*{D\xi}^p + \mathcal{I}(v,k_j) \partial_r (\xi^p) \, dx dr \, .
        \end{aligned}
    \end{equation}

    \noindent We claim that, for $q >1$,
    \begin{equation}
        \label{eq17}
        \frac{q}{2} \biggl( \frac{1}{4} \biggl)^{q-1} (v-k_j)_{-}^2 \leq \mathcal{I}(v,k_j) \leq \frac{q}{2} (v-k_j)_{-}^2 \, .
    \end{equation}
    \noindent In fact, observe that on the set where $\{v<k_j \}$ we have $v < \frac{k_0}{2^j}= \frac{\epsilon_2 (\delta_2 R^p)^{\frac{1}{p-2}}}{2^j}$, so
    \begin{equation*}
        q \int_{v}^{k_j} \biggl(1 - \frac{3}{4} e^{-\frac{r}{p-2}}(\delta_2 R^p)^{-\frac{1}{p-2}}s \biggl)^{q-1} (k_j - s) \, ds \leq q \int_{v}^{k_j} (k_j -s) \, ds = \frac{q}{2} (v-k_j)_{-}^2
    \end{equation*}
    and
    \begin{equation*}
        \begin{aligned}
             & q \int_{v}^{k_j} \biggl(1 - \frac{3}{4} e^{-\frac{r}{p-2}}(\delta_2 R^p)^{-\frac{1}{p-2}}s \biggl)^{q-1} (k_j - s) \, ds \\
             \geq & q \int_{v}^{k_j} \biggl(1 - \frac{3}{4} e^{-\frac{r}{p-2}}(\delta_2 R^p)^{-\frac{1}{p-2}}k_j \biggl)^{q-1} (k_j - s) \, ds \\
             = & q \int_{v}^{k_j} \biggl(1 - \frac{3}{4} e^{-\frac{r}{p-2}}\frac{\epsilon_2}{2^j} \biggl)^{q-1} (k_j - s) \, ds \\
             \geq & q \int_v^{k_j}  \biggl( 1-\frac{3}{4} \biggl)^{q-1} (k_j -s) \, ds = \frac{q}{2} \biggl( \frac{1}{4} \biggl)^{q-1} (v-k_j)_{-}^2 \, .
        \end{aligned}
    \end{equation*}
    It is easily proved that, for $q \in (0,1)$,  the inequalities are the inverse ones 
    \begin{equation*}
        \frac{q}{2}  (v-k_j)_{-}^2 \leq \mathcal{I}(v,k_j) \leq \frac{q}{2}\biggl( \frac{1}{4} \biggl)^{q-1} (v-k_j)_{-}^2 \, .
    \end{equation*}

    \noindent Finally, we want to estimate $I_2$ from above. For that purpose, recall that $e^{-x} \leq  \frac{1}{x}$, for all $ x \geq 0$, therefore 
    \begin{equation*}
        e^{-\frac{r}{p-2}} \leq \frac{p-2}{r} \, .
    \end{equation*}
\begin{enumerate}
    \item For $q>2$, 
    \begin{equation*}
        \begin{aligned}
            I_2 \leq & \frac{3}{4} (\delta_2 R^p)^{-\frac{1}{2-p}} \biggl( \frac{q-1}{p-2} \biggl) q \int_{\theta R^p}^{\theta (8R)^p} \frac{p-2}{r} \int_{K_{8R}} \biggl( \int_v^{k_j} s (k_j-s) \, ds \biggl) \, dx dr \\
             \leq & \frac{3}{4} (\delta_2 R^p)^{-\frac{1}{2-p}} \frac{(q-1)}{\theta (8R)^p}q\int_{\theta R^p}^{\theta (8R)^p} \int_{K_{8R}} k_j \biggl( \int_v^{k_j}  (k_j-s) \, ds \biggl) \, dx dr \\
             \leq & \frac{3}{4} (\delta_2 R^p)^{-\frac{1}{2-p}} \frac{(q-1)}{\theta (8R)^p} q\int_{\theta R^p}^{\theta (8R)^p} \int_{K_{8R}} k_j \biggl( \int_v^{k_j}  (k_j-s) \, ds \biggl) \, dx dr  \\
             \leq & \frac{3}{4} \epsilon_2 \frac{(q-1)}{\theta (8R)^p}  q\int_{\theta R^p}^{\theta (8R)^p} \int_{K_{8R}}   (v-k_j)_{-}^2 \, dx dr \, ;
        \end{aligned}
    \end{equation*}

    \item for $1\leq q \leq 2$,  
    \begin{equation*}
        \begin{aligned}
            I_2 \leq & \frac{3}{4} (\delta_2 R^p)^{-\frac{1}{2-p}} \biggl( \frac{q-1}{p-2} \biggl) q \int_{\theta R^p}^{\theta (8R)^p} \frac{p-2}{r} \int_{K_{8R}} \biggl( k_j \int_v^{k_j} (1- \frac{3}{4}e^{-\frac{r}{p-2}}\frac{\epsilon_2}{2^j})^{q-2} (k_j-s) \, ds \biggl) \, dx dr \\
            \leq &  \frac{3}{4} \epsilon_2 \frac{(q-1)}{\theta (8R)^p} q\int_{\theta R^p}^{\theta (8R)^p} \int_{K_{8R}}  \biggl( \int_v^{k_j} \biggl(\frac{1}{4} \biggl)^{q-2} (k_j-s) \biggl) \, dx dr \\
            \leq & \frac{3}{4} \epsilon_2 \biggl( \frac{1}{4} \biggl)^{q-2} \frac{(q-1)}{\theta (8R)^p} q\int_{\theta R^p}^{\theta (8R)^p} \int_{K_{8R}}   (v-k_j)_{-}^2 \, dx dr \, .
        \end{aligned}
    \end{equation*}

    \item for $0 \leq q \leq 1$,  $I_2 \leq 0$.
    
\end{enumerate}

\noindent In all these cases, we have a bound with a different (but still universal) constant $\gamma \geq 0$ that depends only on the data. Using this together with \eqref{eq17} in \eqref{eq16} we arrive at
    \begin{eqnarray} \label{eq18}
         \int_{\theta R^p}^{\theta (8R)^p} \int_{K_{8R}} \abs*{D(v-k_j)_{-}}^{p}  \xi^p \, dx dr &  \leq & \frac{\gamma}{\theta R^p} \int_{\theta R^p}^{\theta (8R)^p} \int_{K_{8R}} (v-k_j)_{-}^2 \, dx dr  + \gamma \int_{K_{8R} \times\{ \theta R^p\}} (v-k_j)_{-}^2 \, dx  \nonumber  \\
           &  &+ 2\int_{\theta R^p}^{\theta (8R)^p} \int_{K_{8R}} \frac{p^p 2^{p-1}}{R^p} (v-k_j)_{-}^p \, + \frac{(v-k_j)_{-}^2}{\theta (3R)^p} \, dx dr \,  \nonumber \\
           &  =& \biggl[ \frac{\gamma}{R^p} k_j^p + \frac{\gamma}{\theta R^p} k_j^2 \biggl] \abs*{Q \cap \{v < k_j \}} + {k_j^2}\abs*{K_{8R} \cap \{v(\cdot, \theta R^p) \leq k_j \}} \nonumber \\
            & \leq &  \gamma(p,q) \frac{k_j^p}{R^p} \abs*{ Q \cap \{v < k_j \}  }     
    \end{eqnarray}
once we choose  $\theta = \bigl(\frac{k_0}{2^{s_2 -1}} \bigl)^{2-p} \geq k_j^{2-p}$ for all $j= 1, \dots, s_2-1$. We can now apply Lemma \ref{isoperi} with $l=k_j$ and $k=k_{j+1}$ on $K_{4R}$ to obtain
    \begin{equation*}
        {k_j} \abs*{K_{4R} \cap \{v(\cdot,\tau) < k_{j+1} \}} \leq \frac{\gamma R^{N+1}}{\abs*{K_{4R} \cap \{v(\cdot, \tau) > k_j \}}} \int_{K_{4R} \cap  \{k_{j+1}< v(\cdot,\tau) < k_j \}} \abs*{Dv} \, dx \ ,
    \end{equation*}
then we integrate in time over the interval $(\theta (4R)^p, \theta (8R)^p)$
and arrive at
\begin{equation*}
        {k_j} \abs*{\tilde{Q} \cap \{v(\cdot,\tau) < k_{j+1} \}} \leq \gamma \ R \int_{\theta (4R)^p}^{\theta (8R)^p}\int_{K_{4R} \cap  \{k_{j+1}< v(\cdot,\tau) < k_j \}} \abs*{D(v-k_j)_{-}} \, dx d\tau \, .
    \end{equation*}
since  $\abs*{K_{4R} \cap \{v(\cdot, \tau) > k_j \}} \geq \abs*{K_{4R} \cap \{v(\cdot, \tau) > k_0 \}} \geq \frac{1}{4} \abs*{K_{4R}}$. By using  H\"older's inequality together with inequality \eqref{eq18}, we get
    \begin{equation}
        \label{eq20}
        \begin{aligned}
            {k_j} \abs*{\tilde{Q} \cap \{v(\cdot,\tau) < k_{j+1} \}} & \leq \frac{\gamma}{R} \biggl( \int_{\theta (4R)^p}^{\theta (8R)^p}\int_{K_{4R}} \abs*{D(v-k_j)_{-}}^p \, dx d\tau \biggl)^{\frac{1}{p}} \abs*{\tilde{Q} \cap \{k_{j+1} < v < k_j \}}^{\frac{p-1}{p}} \\
            & \leq \frac{\gamma}{R} \biggl( \int_{\theta R^p}^{\theta (8R)^p}\int_{K_{8R}} \abs*{D(v-k_j)_{-}}^p \xi^p \, dx d\tau \biggl)^{\frac{1}{p}} \abs*{\tilde{Q} \cap \{k_{j+1} < v < k_j \}}^{\frac{p-1}{p}} \\
            & \leq \gamma(p,q)\  k_j \ \abs*{\tilde{Q}}^{\frac{1}{p}}\abs*{\tilde{Q} \cap \{k_{j+1} < v < k_j \}}^{\frac{p-1}{p}} \, .
        \end{aligned}
    \end{equation}
 And now, by summing up for $j=0, \cdots, s_2-1$,  
\[ s_2 \abs*{\tilde{Q} \cap [v < k_{s_2}]}^{\frac{p}{p-1}} \leq  \sum_{j=0}^{s_2 -1} \abs*{\tilde{Q} \cap [v<k_{j+1}]}  \leq \sum_{j=0}^{s_2-1} \gamma \abs*{\tilde{Q} \cap [k_{j+1}<v<k_j]} \abs*{\tilde{Q}}^{\frac{1}{p-1}} \leq  \gamma \abs*{\tilde{Q}}^{\frac{p}{p-1}} \, \]
that is, 
       \begin{equation}
        \label{eq22}
         \abs*{\tilde{Q} \cap [v < k_{s_2}]} \leq \frac{\bar{\gamma}}{s_2^{\frac{p-1}{p}}} \abs*{\tilde{Q}} \, .
    \end{equation}
Next step is to pass from this measure type information to a pointwise information via a DeGiorgi type argument, fixing thereby $s_2$. So we 
\begin{enumerate}
    \item construct a sequence of nested and shrinking cylinders \[Q_n^{-} = K_{R_n} \times (\theta (8R)^p-\theta (R_n)^p, \theta (8R)^p) \subset \tilde{Q}\] where $2R < R_n = 2R\biggl( 1 + \frac{1}{2^n} \biggl) \leq 4R$,
    
    \item consider the family cut-off functions $\xi_n \in C^1_0 (Q_n^{-})$ such that 
    \[ 0<\xi_n<1, \quad \xi_n \equiv 1 \ \  \text{in} \  Q_{n+1}^{-}, \quad  \abs*{D\xi} \leq \frac{2^{n+1}}{R}, \quad  \abs*{\partial_r \xi} \leq \gamma \frac{2^{np}}{\theta R^p} \]

    \item proceed as before now with the test functions $(v-k_n)_-\xi_n^p$, where
    \[\frac{k_0}{2^{s_2+1} }< k_n = \frac{k_0}{2^{s_2 + 1}} \biggl( 1+ \frac{1}{2^n}\biggl) < \frac{k_0}{2^{s_2}} \]   
\end{enumerate}  
to  arrive at
    \begin{equation}
        \label{eq23}
        \int_{\theta (8R)^p -\theta R_n^p}^{\theta(8R)^p} \int_{K_{R_n}} \abs*{D(v-k_n)_{-}}^p \xi_n^p \, dx dr \leq \gamma \frac{k_n^p}{R_n^p} \abs*{Q_n^{-} \cap [v < k_n]} = \gamma \frac{k_n^p}{R_n^p} \abs*{A_n} \, .
    \end{equation}
 With the use of H\"older's inequality, Lemma \ref{sobemb} and inequality \eqref{eq23} we find
    \begin{equation*}
        \begin{aligned}
            (k_n - k_{n+1}) \abs*{A_{n+1}} & \leq \iint_{Q_{n+1}^{-}} (v-k_n)_{-} dx dr \leq \iint_{Q_n^{-}} (v-k_n)_{-} \xi_n^p \, dx dr \\
            & \leq \biggl[\iint_{Q_n^{-} } ((v-k_{n})_{-} \xi_n)^{p (\frac{N+2}{N})} \, dx dr \biggl]^{\frac{N}{p(N+2)}} \abs*{A_n}^{1- \frac{N}{p(N+2)}}\\
            & \leq \gamma \biggl[ \biggl( \iint_{Q_n^{-} } \abs*{D((v-k_{n})_{-} \xi_n)}^{p} \, dx dr \biggl) \biggl(\int_{B_{R_n}} [(v-k_n)_{-} \xi_n]^2 \, dx \biggl)^{\frac{p}{N}}\biggl]^{\frac{N}{p(N+2)}} \\
            & \quad \times \abs*{A_n}^{1-\frac{N}{p(N+2)}} \\
            & \leq \gamma \biggl[ \biggl( \iint_{Q_n^{-} } \abs*{D(v-k_{n})_{-}}^p \xi_n^{p} \, dx dr + \iint_{Q_n^{-}} (v-k_n)_{-}^p \abs*{D \xi_n}^p \biggl)\\
            & \quad \times \biggl(\int_{K_{R_n}} [(v-k_n)_{-} \xi_n]^2 \, dx \biggl)^{\frac{p}{N}}\biggl]^{\frac{N}{p(N+2)}} \abs*{A_n}^{1-\frac{N}{p(N+2)}}  \\
            & \leq \gamma \biggl[\biggl(\frac{k_n^p}{R_n^p} \biggl) \biggl( \frac{k_n^{2}}{\theta R_n^p} \biggl)^{\frac{p}{N}} \abs*{A_n}^{1+\frac{p}{N}}  \biggl]^{\frac{N}{p(N+2)}} \abs*{A_n}^{1-\frac{N}{p(N+2)}} \\
            & \leq \gamma \frac{k_n}{(\theta R_n^p R_n^N )^{\frac{1}{N+2}}} \abs*{A_n}^{1 + \frac{1}{N+2}} \, .
        \end{aligned}
    \end{equation*}
which leads to 
 \begin{equation}
        \label{eq25}
        Y_{n+1} \leq \gamma \ 2^n \ Y_n^{1+\frac{1}{N+2}} \ , \text{for} \ \ \qquad Y_n := \frac{\abs*{A_{n}}}{\abs*{Q_n^{-}}}
    \end{equation}
and therefore, by Lemma \ref{FCL}, 
 \begin{equation}
        \label{eq26}
        v(x,r) \geq \frac{k_0}{2^{s_2 +1}} \qquad \text{in $K_{2R} \times (\theta (6R)^p, \theta (8R)^p)$} \, .
    \end{equation}
once we choose $s_2$ big enough so that 
    \begin{equation*}
        s_2^{\frac{p-1}{p}} \geq \bar{\gamma} \gamma^{\frac{1}{\alpha}}b^{\frac{1}{\alpha^2}} \, .
    \end{equation*}
Inequality \eqref{eq26} transforms into 
    \begin{equation*}
        u(x,t) \geq \biggl( \frac{3 }{4} \omega \biggl) e^{- \bigl(\frac{\epsilon_2}{2^{s_2}} \bigl)^{2-p}\frac{8^p}{\delta_2(p-2)}} \frac{\epsilon_2}{2^{s_2 +1}}:= \eta_2 \frac{3}{4} \omega \qquad  \text{in $K_{2R}$} 
    \end{equation*}
    for almost all
    \begin{equation*}
        s + \biggl(\frac{3}{4} \biggl)^{2-p} \omega^{1+q-p} \delta_2 e^{ \bigl(\frac{\epsilon_2}{2^{s_2}} \bigl)^{2-p} \frac{6^p}{\delta_2}}R^p \leq t \leq s + \biggl(\frac{3}{4} \biggl)^{2-p} \omega^{1+q-p} \delta_2 e^{ \bigl(\frac{\epsilon_2}{2^{s_2}} \bigl)^{2-p}\frac{8^p}{\delta_2}}R^p \, ,
    \end{equation*}
completing the proof.
    \end{proof}

    \section{H\"older Continuity}

    \noindent In the previous section, we showed that for each one of the alternatives \eqref{alt1} and \eqref{alt2} we can derive a pointwise information of the solution $u$ in an {\it intrinsic} cylinder. We will use this information to prove the local H\"older continuity. 

    \vspace{.2cm}

    \noindent The first thing we want to show is that, by correctly choosing the constants $\epsilon_1$, $\delta_1$, $\epsilon_2$, $\delta_2 \in (0,1)$ and $s_1$, $s_2 \in \N$ coming from expansion of positivity results, we can construct a cylinder (more precisely a time interval) that combines the results coming from the study of the two alternatives within which 
    \begin{equation*}
        u \geq \min\left\{\frac{\eta_1}{4}, \frac{3 \, \eta_2}{4} \right\}  \omega =: \eta_{*}\omega \, .
    \end{equation*}
We denote
    \begin{equation*}
        \begin{aligned}
        &(s+a_1 \omega^{q+1-p} R^p,s+b_1\omega^{q+1-p} R^p):= \\
        &\biggl( s  + \delta_1 \biggl(\frac{\omega}{4}\biggl)^{q+1-p}e^{\frac{6^p}{\delta_1} \bigl(\frac{\epsilon_1}{2^{s_1}} \bigl)^{q+1-p}} R^p, s + \delta_1 \biggl(\frac{\omega}{4}\biggl)^{q+1-p}e^{\frac{8^p}{\delta_1}\bigl(\frac{\epsilon_1}{2^{s_1}} \bigl)^{q+1-p}} R^p\biggl) \\
        \text{and} \\
        &(s+a_2 \omega^{q+1-p} R^p,s+b_2\omega^{q+1-p}R^p) := \\
        &\biggl(s + \biggl(\frac{3}{4} \biggl)^{2-p} \omega^{q+1-p} \delta_2 e^{ \bigl(\frac{\epsilon_2}{2^{s_2}} \bigl)^{2-p} \frac{6^p}{\delta_2}} R^p, s + \biggl(\frac{3}{4} \biggl)^{2-p} \omega^{q+1-p} \delta_2 e^{ \bigl(\frac{\epsilon_2}{2^{s_2}} \bigl)^{2-p}\frac{8^p}{\delta_2}} R^p\biggl) \, .
        \end{aligned}
    \end{equation*}
     and call $a:= \max\{a_1,a_2 \}$ and $b:= \min\{b_1,b_2 \}$. Observe that if $(a= a_1 \, \wedge \,b=b_1 )$ or $(a= a_2  \, \wedge \,  b=b_2)$ we would find that one of the intervals is entirely contained in the other. Therefore the non trivial cases to study are either $(a=a_1 \, \wedge b=b_2)$ or $(a=a_2 \,  \wedge b=b_1)$. In the first case we need to assure that, by appropriate choices one has  $a_1<b_2$; while in the second case we need to prove that we can choose $a_2 < b_1$ so that intervals intersection occurs, as shown in the following picture
    
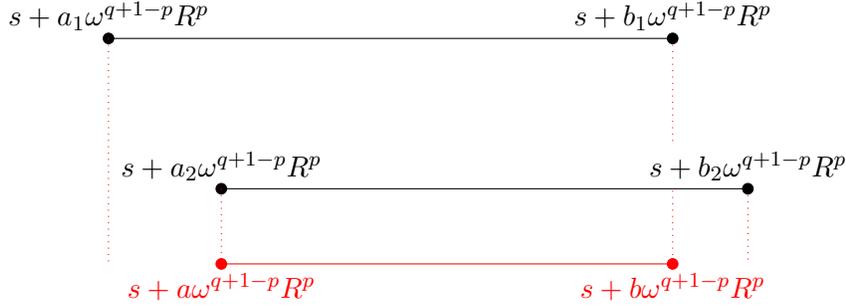
\begin{center}
    \begin{tikzpicture}
        \draw (-4.5,1) -- (3,1);
        \path (-6,0) -- (6,0);
        \path (0,3) -- (0,-3);
        \draw[dotted, red] (3,1) -- (3,-2);
        \draw[dotted, red] (4,-1) -- (4,-2);
        \node at (4, -0.7)[fill=white] {$s + b_2 \omega^{q+1-p} R^p$};
        \draw (-3,-1) -- (4,-1);
        \filldraw (-4.5,1) circle (2pt);
        \filldraw (3,1) circle (2pt);
        \node at (-4.5, 1.3) {$s + a_1 \omega^{q+1-p} R^p$};
        \node at (3, 1.3) {$s + b_1 \omega^{q+1-p} R^p$};
        \filldraw (-3,-1) circle (2pt);
        \filldraw (4,-1) circle (2pt);
        \node at (-3, -0.7) {$s + a_2 \omega^{q+1-p} R^p$};
        
        \draw[dotted, red] (-4.5,1) -- (-4.5,-2);
        \draw[dotted, red] (-3,-1) -- (-3,-2);
        \draw[red] (-3,-2) -- (3,-2);
        \filldraw[red] (-3,-2) circle (2pt);
        \filldraw[red] (3,-2) circle (2pt);
        \node at (-3,-2.3)[red] {$s+a \omega^{q+1-p}R^p$};
        \node at (3,-2.3) [red] {$s+b \omega^{q+1-p} R^p$};
    \end{tikzpicture}
\end{center}

\noindent We commence by considering that $(a=a_1 \, \wedge \,  b=b_2)$. Once we choose
    \begin{equation*}
        \delta_1 < \delta_2 \biggl( \frac{1}{4} \biggl)^{p-q-1} \biggl(\frac{3}{4} \biggl)^{2-p}
    \end{equation*}
    and
    \begin{equation*}
        \epsilon_2^{p-2} <  \biggl( \frac{8}{6} \biggl)^{p} \biggl(  \frac{\epsilon_1}{2^{s_1}} \biggl)^{p-q-1} \biggl( \frac{\delta_1}{\delta_2} \biggl) (2^{s_2})^{p-2}\, 
    \end{equation*}
    we get 
    \begin{equation*}
             \delta_1 \biggl(\frac{1}{4} \biggl)^{q+1-p} e^{\frac{6^p}{\delta_1}\bigl( \frac{\epsilon_1}{2^{s_1}} \bigl)^{q+1-p}} <\biggl( \frac{8}{6} \biggl)^{p}  \delta_2 \biggl(\frac{3}{4} \biggl)^{2-p} e^{\frac{8^p}{\delta_2}\bigl( \frac{\epsilon_2}{2^{s_2}} \bigl)^{2-p}}  \ 
    \end{equation*}
 and therefore $a_1<b_2$. On the other hand, having $a_2 <b_1$ means that
    \begin{equation*}
        \delta_2 \biggl(\frac{3}{4} \biggl)^{2-p} e^{\frac{6^p}{\delta_2}\bigl( \frac{\epsilon_2}{2^{s_2}} \bigl)^{2-p}} < \delta_1 \biggl(\frac{1}{4} \biggl)^{q+1-p} e^{\frac{8^p}{\delta_1}\bigl( \frac{\epsilon_1}{2^{s_1}} \bigl)^{q+1-p}} \, ,
    \end{equation*}
so, one just needs to choose
    \begin{equation*}
        \delta_2 < \delta_1 \biggl( \frac{1}{4} \biggl)^{q+1-p} \biggl(\frac{3}{4} \biggl)^{p-2}
    \end{equation*}
    and
    \begin{equation*}
        \epsilon_1^{p-q-1} < \biggl( \frac{8}{6} \biggl)^{p} \biggl(  \frac{\epsilon_2}{2^{s_2}} \biggl)^{p-2} \biggl( \frac{\delta_2}{\delta_1} \biggl) (2^{s_1})^{p-q-1} \, .
    \end{equation*}

 \vspace{.2cm}
 
 \noindent Under the previous choices, one obtains 
    \begin{equation*}
        u \geq \eta_{*} \omega \qquad \text{in $K_{2R}(x_0) \times (s+a \omega^{q+1-p}R^p,s+b\omega^{q+1-p} R^p)$} \, 
    \end{equation*}
and by choosing $s=t_0- b \omega^{q+1-p} R^p$ and taking $c= b-a$, we arrive at
\begin{equation*}
u \geq \eta_{*} \omega \qquad \text{in} \  K_{2R}(x_0) \times (t_0 - c \omega^{q+1-p}R^p,t_0)
    \end{equation*}
hence,  
\begin{equation*}
    \essosc_{K_R(x_0) \times (t_0 - c \omega^{q+1-p} R^p,t_0)} u\leq  (1- \eta_{*}) \omega =: \bar{\eta} \omega \, .
    \end{equation*} 

\vspace{.2cm}

  \begin{remark}
    \label{rmk}
      Observe that both $a$, $b \geq 1$. In fact, if for example we consider $b =b_2 < 1$, then
        \begin{equation*}
            \begin{aligned}
                &\biggl(\frac{3}{4} \biggl)^{2-p} \delta_2 e^{ \bigl(\frac{\epsilon_2}{2^{s_2}} \bigl)^{2-p} \frac{8^p}{\delta_2}} < 1 \longleftrightarrow
                 e^{ \bigl(\frac{\epsilon_2}{2^{s_2}} \bigl)^{2-p} \frac{8^p}{\delta_2}} < \biggl( \frac{3}{4} \biggl)^{p-2} \frac{1}{\delta_2} \Leftrightarrow
                 \bigl(\frac{\epsilon_2}{2^{s_2}} \bigl)^{2-p} 8^p < \delta_2 \ln \biggl( \biggl( \frac{3}{4} \biggl)^{p-2} \frac{1}{\delta_2} \biggl) < 1 \, ,
            \end{aligned}
        \end{equation*}
        which is an absurd. The other cases can be treated in a similar way.
\end{remark}
        \vspace{.2cm}
        
    \noindent We are now in position to prove the following reduction of oscillation.
    \begin{theorem}
    \label{osc}
        Let $u$ be a non negative, locally bounded, local weak solution to \eqref{EQ} in $\Omega_T$ and let $(x_0,t_0)$ be an interior point of the domain. Then there exist positive constants $\epsilon_0$, $\alpha \in (0,1)$ and $b$, $\gamma >1$, depending only on the data, such that, for every radius $0<r<R$,
        \begin{equation}
            \label{eqosc}
            \essosc_{K_r(x_0) \times (t_0- \bar{\omega}^{q+1-p} r^p,t_0)} \leq \gamma \  \bar{\omega} \  \biggl(\frac{r}{R} \biggl)^{\alpha} \, ,
        \end{equation}
        where $\bar{\omega}= \max\{\omega, (2b)^{\frac{1}{p-q-1}}R^{\frac{\epsilon_0}{p-q-1}} \}$.
    \end{theorem}
    \begin{proof}
        The proof of this result relies on an iterative scheme which will be presented into several steps as follows.
        
        \vspace{0.5 cm}
    
        \noindent \textbf{STEP 1}: Consider the two cylinders introduced in Section 2 
        \[ \tilde {Q}=K_{8R}(x_0) \times (t_0- R^{p-\epsilon_0},t_0) \quad \text{and} \quad Q=K_R(x_0) \times (t_0-A \omega^{q+1-p} R^p,t_0) \ \] and assume that   $Q \subset \bar{Q}$, meaning  
        \begin{equation}
            \label{eq28}
            \omega > A^{\frac{1}{p-q-1}} R^{\frac{\epsilon_0}{p-q-1}} \, .
        \end{equation}
        Then we have 
        \[  \essosc_{Q} u \leq \essosc_{\tilde{Q}}  u = \omega  \]
        and this is the starting point of the procedure of get the reduction of the oscillation: by procceding as described in Section 3 and taking $A=2b$, we arrive at
        \begin{equation}
            \label{redosc}
            \essosc_{K_R(x_0) \times (t_0 - c \omega^{q+1-p} R^p,t_0)} u\leq  (1- \eta_{*}) \omega =: \bar{\eta} \omega
        \end{equation}
        
        \vspace{0.5 cm}
        
        \noindent \textbf{STEP 2}: Define $\omega_1 := \max\{\bar{\eta} \omega, (2b)^{\frac{1}{p-q-1}}R^{\frac{\epsilon_0}{p-q-1}} \}$ and $R_1: = \lambda R$ for $\lambda := (2b)^{-\frac{1}{p}} \ \bar{\eta}^{\frac{p-q-1}{p}} \ c ^{\frac{1}{p}} \in (0,1)$. Then 
        \begin{equation}
            \label{eq30b}
            \begin{aligned}
                K_{R_1}(x_0) \times (t_0-(2b) \omega_1^{q+1-p} R_1^p ,t_0) \subset K_R(x_0) \times (t_0- c\omega^{q+1-p} R^p,t_0) \, .
            \end{aligned}
        \end{equation}
        and therefore, due to \eqref{redosc}, 
        \begin{equation}
            \label{eq31}
            \essosc_{K_{R_1}(x_0) \times (t_0 -(2b) \omega_1^{q+1-p} R_1^p ,t_0)} u \leq \essosc_{K_R (x_0) \times (t_0 - c\omega^{q+1-p} R^p,t_0) } u \leq \bar{\eta} \omega \leq \omega_1 \, .
        \end{equation}
        Observe that if \eqref{eq28} fails then, by recalling the definition of $\omega_1$, 
        \begin{equation*}
            (2b)\omega_1^{q+1-p} R_1^p \leq \lambda^p R^{p-\epsilon_0}\leq R^{p-\epsilon_0}
        \end{equation*}
        and we get
        \begin{equation}
        \essosc_{K_{R_1}(x_0) \times (t_0-(2b) \omega_1^{q+1-p} R_1^p ,t_0)} u \leq  \omega \leq (2b)^{\frac{1}{p-q-1}} R^{\frac{\epsilon_0}{p-q-1}} \leq \omega_1 \, .
        \end{equation}
        
        \vspace{0.5 cm}
        
        \noindent \textbf{STEP 3}: Define for, $i \in \N$,
        \begin{equation*}
        \begin{aligned}
            R_i = \lambda^i R, \quad \omega_0 &= \omega , \quad \omega_i = \max\{\bar{\eta}\omega_{i-1}, (2b)^{\frac{1}{p-q-1}} R_{i-1}^{\frac{\epsilon_{0}}{p-q-1}} \} \\
            &Q_i = K_{R_{i}(x_0)} \times (t_0-(2b) \omega_i^{q-1-p}R_i^p,t_0) \, 
            \end{aligned}
        \end{equation*}
        where $\left\{Q_i\right\}_{i\in \N}$ defines a sequence of nested and shrinking cylinders since 
        \begin{equation*}
            (2b)\omega_{n+1}^{q+1-p} R_{n+1}^p \leq (2b) \bar{\eta}^{q+1-p}\omega_n^{q+1-p} \lambda^p R_n^p =  c \omega_n^{q+1-p} R_n^p \leq (2b) \omega_n^{q+1-p} R_n^p  \ . 
        \end{equation*}
        We want to prove (by induction) that, $\forall \, i \in \N$,
        \begin{equation}
            \label{eq32}
            \essosc_{Q_i} \leq \omega_i \, ,
        \end{equation}
       The case $n=1$ is precisely \eqref{eq31}. Suppose now that \eqref{eq32} holds for $i=1,\dots,n$. Consider that, for some time level $s$, either 
        \begin{equation*}
            \abs*{K_{R_n}(x_0) \cap \{u(x,s) > \frac{\omega_n}{4}\} }\geq \frac{1}{2}\abs*{K_{R_n}(x_0)}
        \end{equation*}
       or
        \begin{equation*}
            \abs*{K_{R_n}(x_0) \cap \{u(x,s) > \frac{\omega_n}{4}\} }\leq \frac{1}{2}\abs*{K_{R_n}(x_0)} \, .
        \end{equation*}
       which corresponds to \eqref{alt1} or \eqref{alt2}, with $R=R_n$ and $\omega=\omega_n$. By arguing as in Section 3 (note that  all the constants $\epsilon_{1}$, $\epsilon_2$, $\delta_1$, $\delta_2$, $s_1$, $s_2$, $\eta_1$ and $\eta_2$ remain the same, because they depend only on the data), we find
        \begin{equation}
        \label{eq33}
            u \geq \eta_{*} \omega_n \qquad \text{in $K_{2R_n} (x_0)\times (s+a\omega_n^{q+1-p} R_n^p, s+b\omega_n^{q+1-p} R_n^p)$} \, ,
        \end{equation}
        for some $s$, that is to be chosen. 
        
        \noindent On the time interval $(t_0 -(2b) \omega_n^{q+1-p}R_n^p,t_0) $, we have the following dichotomy, either
        \begin{equation}
            \label{dicn1}
            (2b)\omega_n^{q+1-p}R_n^p < R_n^{p-\epsilon_0} \, ,
        \end{equation}
        or
        \begin{equation}
            \label{dicn2}
            (2b)\omega_n^{q+1-p}R_n^p \geq R_n^{p-\epsilon_0} \, .
        \end{equation}

        \vspace{0.5 cm}
        \noindent \textbf{STEP 4}: Assume that \eqref{dicn1} holds. Then $K_{R_n}(x_0) \times(t_0-(2b)\omega_n^{q+1-p}R_n^p,t_0) \subset K_{R_n}(x_0) \times (t_0-R_n^{p-\epsilon_0},t_0) \subset \tilde{Q} \subset \Omega_T$. Reasoning as before we take $s= t_0 + b \omega_n^{q+1-p} R_n^p$ and $c=b-a$, and then 
        \begin{equation*}
            Q_{n+1} \subset K_{R_{n}(x_0)} \times (t_0 -c\omega_n^{q+1-p} R_n^p,t_0) \subset K_{2R_n}(x_0) \times (s+a\omega_n^{q+1-p} R_n^p, s+b\omega_n^{q+1-p} R_n^p)
        \end{equation*}
        and 
        \begin{equation}
            \label{eq34}
            \begin{aligned}
            \essosc_{Q_{n+1}} u &\leq \essosc_{K_{R_n} (x_0)\times (t_0-c \omega_n^{q+1-p}R_n^p,0)} u \leq (1- - \eta_{*} )\omega_n = \bar{\eta} \omega_n \leq \omega_{n+1} \, .
            \end{aligned}
        \end{equation}
        On the other hand, if \eqref{dicn2} holds, then $Q_{n} \subset K_{R_n}(x_0)  \times (t_0-R_n^{p-\epsilon_0},t_0) \subset \tilde{Q}$, and so
        \begin{equation*}
            \essosc_{Q_{n+1}} u \leq \essosc_{Q_n} u \leq \omega_n \leq (2b)^{\frac{1}{p-q-1}} R_n^{\frac{\epsilon_0}{p-q-1}} \leq \omega_{n+1} \, .
        \end{equation*}
        This proves \eqref{eq32}. 

        \vspace{0.5 cm}
        
         \noindent \textbf{STEP 5}: We now get to the final step of the proof: we prove  \eqref{eqosc} once we properly choose $\epsilon_0$ and fix $n \in \N$.
        
        \noindent From the definition of $\omega_n$ and by choosing $\epsilon_0$ so small, so that $\bar{\eta} \leq \big( \frac{\lambda}{2}\big)^{\frac{\epsilon_0}{p-q-1}}$, we get 
        \begin{eqnarray}\label{eq35}
            \omega_{n+1} & \leq & \bar{\eta} \omega_n + (2b)^{\frac{1}{p-q-1}} R_n^{\frac{\epsilon_0}{p-q-1}} \leq \dots \nonumber \\
            & \leq &  \bar{\eta}^n \omega  + (2b)^{\frac{1}{p-q-1}} R^{\frac{\epsilon_0}{p-q-1}}\sum_{i=0}^{n-1} \bar{\eta}^{i} \lambda^{(n-i)\frac{\epsilon_0}{p-q-1}} \nonumber  \\
            & \leq & \bar{\eta}^n \omega + (2b)^{\frac{1}{p-q-1}} R^{\frac{\epsilon_0}{p-q-1}}  \lambda^{n \frac{\epsilon_0}{p-q-1}} \sum_{i=0}^{n-1}  2^{-i \frac{\epsilon_0}{p-q-1}} \nonumber  \\
            \omega_{n+1} & \leq & \bar{\eta}^n \omega + \gamma (2b)^{\frac{1}{p-q-1}} R^{\frac{\epsilon_0}{p-q-1}} \lambda^{n \frac{\epsilon_0}{p-q-1}} \ ,  \qquad \gamma= \gamma(p,q,N) \ .
        \end{eqnarray}
Now, take $r \in (0,R)$ and fix $n \in \N$ big enough so that $\lambda^{n+2}R<r<\lambda^{n+1} R$. Then, by defining $\alpha_1 = \frac{\ln(\bar{\eta})}{\ln(\lambda)} \in (0,1)$ we get
        \begin{equation*}
            \bar{\eta}^{n} \leq \bar{\eta}^{n} \biggl(  \frac{r}{\lambda^{n+2}R}\biggl)^{\alpha_1} = \frac{1}{\bar{\eta}^2} \biggl( \frac{r}{R}\biggl)^{\alpha_1} \, .
        \end{equation*}
        Moreover, using $\lambda^{n+2}<\frac{r}{R}$
        \begin{equation*}
             \gamma (2b)^{\frac{1}{p-q-1}} R^{\frac{\epsilon_0}{p-q-1}} \lambda^{n \frac{\epsilon_0}{p-q-1}} \leq \frac{\gamma}{\lambda^{\frac{2 \epsilon_0}{p-q-1}}} R^{\frac{\epsilon_0}{p-q-1}} (2b)^{\frac{1}{p-q-1}} \biggl( \frac{r}{R}\biggl)^{\frac{\epsilon_0}{p-q-1}} \, .
        \end{equation*}
        Combining these two inequalities together with \eqref{eq32} and \eqref{eq35} and defining $\bar{\omega} : = \max\{\omega, (2b)^{\frac{1}{p-q-1}} R^{\frac{\epsilon_0}{p-q-1}}\} > \omega_n$ , for all $ n \in \N$, and taking $\alpha = \min\{\alpha_1, \frac{\epsilon_0}{p-q-1}\}$
        \begin{equation*}
        \begin{aligned}
            \essosc_{K_r (x_0)\times (t_0-\bar{\omega}^{q+1-p}r^p,t_0)} u & \leq \essosc_{K_{R_{n+1}}(x_0) \times (t_0-\omega_{n+1}^{q+1-p}R_{n+1}^p,t_0)} u \leq \essosc_{Q_{n+1}} u \leq \omega_{n+1} \\
            & \leq \frac{\omega}{\bar{\eta}^2} \biggl( \frac{r}{R}\biggl)^{\alpha_1} + \frac{\gamma}{\lambda^{\frac{2\epsilon_0}{p-q-1}}} R^{\frac{\epsilon_0}{p-q-1}} (2b)^{\frac{1}{p-q-1}} \biggl( \frac{r}{R}\biggl)^{\frac{\epsilon_0}{p-q-1}} \\
            & \leq \left[\frac{1}{\bar{\eta}^2} \biggl( \frac{r}{R}\biggl)^{\alpha_1} + \frac{\gamma}{\lambda^{\frac{2\epsilon_0}{p-q-1}}}  \biggl( \frac{r}{R}\biggl)^{\frac{\epsilon_0}{p-q-1}} \right] \bar{\omega} \\
            & \leq  \bar{\gamma} \,\bar{\omega} \biggl( \frac{r}{R}\biggl)^{\alpha}
            \end{aligned}
        \end{equation*}
        
    \end{proof}

    \begin{remark}
    \label{dibene}
        We give a simple yet useful observation: as in the $p$-Laplacian case, the proof of Theorem \ref{osc} continues to work if we just know that the intrinsic cylinder satisfy
        \begin{equation*}
            K_R(x_0) \times (t_0 - (2b) \omega^{q+1-p} R^p, t_0) \subset \Omega_T \, .
        \end{equation*}

        \noindent The role of the big cylinder $K_{8R}(x_0) \times (t_0-R^{p-\epsilon_0},t_0)$ is just to make sure that in the intrinsic cylinders it holds a control of the oscillation
        \begin{equation*}
            \essosc_{K_{R}(x_0) \times (t_0-(2b)\omega^{q+1-p}R^p,t_0)} u \leq \omega \, .
        \end{equation*}
    \end{remark}

    \noindent The previous theorem allows us to prove the main result.
    \begin{proof}[{\bf Proof of Theorem \ref{holder}}]
        As it it well known, Theorem \ref{osc} allows to show that $u$ is locally H\"older continuous.  To be more precise, let $U_{\tau}$ be a compact subset of $\Omega_T$, call $M:= \norm*{u}_{L^{\infty}(\Omega_T)}$ and fix $(x_1,t_1)$, $(x_2,t_2) \in U_{\tau}$, with $t_2 > t_1$. Define the intrinsic parabolic distance of $U_{\tau}$ to the parabolic boundary of $\Omega \times (0,T]$ as
        \begin{equation*}
        \begin{aligned} 
            (p,q)-\text{dist}:= \inf_{\substack{(y,s) \in \partial_p \Omega \times (0,T]\\ (x,t) \in U_{\tau}}
            } \{\abs*{x-y} + \frac{M^{\frac{p-q-1}{p}}}{(2b)^{\frac{1}{p}}}\abs*{t-s}^{\frac{1}{p}} \} = 2R \ , 
            \end{aligned}
        \end{equation*}
       \[\bar{M}: = \max \{M, (2b)^{\frac{1}{p-q-1}} (2R)^{\frac{\epsilon_0}{p-q-1}} \}\] and construct the cylinder 
        \begin{equation}
        \label{cond11}
            K_R (x_2) \times (t_2 - (2b)M^{q+1-p} R^p,t_2) \subset \Omega_T \, .
        \end{equation}

\noindent The proof of the H\"older continuity will be performed separately in time and space.

\vspace{.1cm}

        \noindent Assume $t_2 - t_1\geq \bar{M}^{q+1-p}R^p$, then
        \begin{equation*}
            \abs*{u(x_2,t_1) - u(x_2,t_2)} \leq 2M \leq 4\bar{M} \biggl(\frac{R}{(p,q)-\text{dist}} \biggl)^{\alpha} \leq 4\bar{M} \biggl( \frac{\bar{M}^{\frac{p-q-1}{p}}\abs*{t_1-t_2}^{\frac{1}{p}}}{(p,q)-\text{dist}} \biggl)^{\alpha} \, .
        \end{equation*}
        If  $t_2 - t_1 < \bar{M}^{q+1-p}R^p$ consider $\abs*{t_2-t_1} =: \bar{M}^{q+1-p} r^p$. Thanks to condition \eqref{cond11} and Remark \ref{dibene} we can apply Theorem \ref{osc} to have
        \begin{equation*}
            \abs*{u(x_2,t_1)-u(x_2,t_2)} \leq \essosc_{K_r(x_2) \times (t_2-\bar{M}^{q+1-p}r^p,t_2)} u \leq \bar{\gamma} \bar{M} \biggl(\frac{\bar{M}^{p-q-1}\abs*{t_2-t_1}^{\frac{1}{p}}}{(p,q)-\text{dist}} \biggl)^{\alpha} \, .
        \end{equation*}

        \noindent As for the H\"older continuity in space, we start by  considering that $\abs*{x_1-x_2} < R$, therefore $\abs*{x_1-x_2} = r$, $0<r<R$ Then, by Theorem \ref{osc}
        \begin{equation*}
            \abs*{u(x_1,t_2) - u(x_2,t_2)} \leq \essosc_{K_r(x_2) \times (t_2-\bar{M}^{q+1-p}r^p,t_2)} u \leq \bar{\gamma} \bar{M} \biggl(\frac{\abs*{x_1-x_2}}{(p,q)-\text{dist}} \biggl)^{\alpha} \, .
        \end{equation*}
        Finally, if $\abs*{x_1-x_2} \geq R$,
        \begin{equation*}
            \abs*{u(x_1,t_2) - u(x_2,t_2)} \leq 2M \leq 4\bar{M} \biggl(\frac{R}{(p,q)-\text{dist}} \biggl)^{\alpha} \leq 4\bar{M} \biggl( \frac{\abs*{x_1-x_2}}{(p,q)-\text{dist}} \biggl)^{\alpha} \, .
        \end{equation*}
        
    \end{proof}


    \section*{APPENDIX: MOLLIFICATION}

    \noindent Along the proofs of Lemma \ref{LemmaEE} and Proposition \ref{CLU2}, we justified the usage of certain test functions by considering a mollified version of the weak formulation in which these were admissible test functions. Such a procedure is widely known in parabolic regularity theory (see for instances \cite{BDL} or \cite{Henriques 2022}) and is formally justified by the following reasoning.

    \vspace{.2cm}
    
    \noindent For all $v :  \Omega_T \to \R$ and all $h \in (0,1)$, we define a finite convolution of $v$ with exponential weight functions by setting
    \begin{equation*}
        [v]_h (x,t) = \int_0^t \frac{v(x,\tau)}{h} e^{\frac{\tau-t}{h}} d\tau \ , \qquad (x,t) \in \Omega_T
    \end{equation*}
    which satisfy the following properties (see  \cite{KL} for details).
    \begin{proposition}
    \label{propmol}
        Let $v \in L^q (\Omega_T)$ for some $q \in [1, +\infty)$. Then for all $h \in (0,1)$ we have:
        \begin{enumerate}
            \item{$[v]_h$ is differentiable in $t$ with
            \begin{equation*}
                \frac{\partial [v]_h}{\partial t} (x,t)= \frac{v(x,t)-[v]_h(x,t)}{h} ;
            \end{equation*}}
            \item{$[v]_h \in L^q (\Omega_T)$ and $[v]_h \to v$ in $L^q(\Omega_T)$ as $h \to 0^+$; }
            \item{If, in addition, $Dv \in L^q(\Omega_T)$, then $[Dv]_h = D[v]_h$ componentwise and $[Dv]_h \to Dv$ in $L^q (\Omega_T)$ as $h \to 0^+$;}
        \end{enumerate}
    \end{proposition}

    \noindent The following lemma provides a rigorous proof of the formulation \eqref{mollified1} in Lemma \ref{LemmaEE}. Without loss of generality, we provide a proof for weak solution to \eqref{EQ}, but the reasoning can be adapted to include the case of equation \eqref{EQ2}.
    \begin{lemma}
    \label{mollification1}
        Let $u$ be a non negative, local weak solution to \eqref{EQ}. Then, for all $\phi \in L^p_{\text{loc}}(0,T; H_0^{1,p}(K))$ we have
        \begin{equation}
           \int_0^T \int_K \frac{\partial}{\partial t}[u^q]_h(x,t) \phi(x,t) \, dx dt + \int_0^T \int_K [(Du)^{p-1}]_h D\phi \, dx dt = \int_K u^q (x,0) \biggl(\int_0^T \frac{\phi(x,t)e^{-\frac{s}{h}}}{h} \, ds \biggl) \, dx \, .
        \end{equation}
    \end{lemma}
    \begin{proof}
        Write the weak formulation \eqref{weak-dfn} for $[\phi]_l (x,\tau+s)$, where $\phi$ is a test function, between $0$ and $T-s$, for some $s<T$,
        \begin{equation}
            \label{m1}
            \int_K u^q(x,\tau) [\phi]_l (\tau+s)\, dx \biggl|_{0}^{T-s} + \int_0^{T-s} \int_K \abs*{Du}^{p-2}Du (x,\tau) \cdot [D\phi]_l(x,\tau+s) - u^q(x,\tau) \partial_\tau [\phi]_{l} (x,\tau+s) \, dx d\tau = 0 \, .
        \end{equation}
        Now consider the change of variable $\tau=t-s$
        \begin{equation}
        \label{m2}
        \begin{aligned}
            &\int_s^{T} \int_K \abs*{Du}^{p-2}Du (x,t-s) \cdot [D\phi]_l(x,t) - u^q(x,t-s) \partial_t [\phi]_l (x,t) \, dx dt \\
            + & \int_K u^q(x,T-s) [\phi]_l(x,T)\, dx = \int_K u^q(x,0) [\phi]_l(x,s) \, dx \, .
        \end{aligned}
        \end{equation}
        from which, by multiplying both sides by $\frac{e^{-\frac{s}{h}}}{h}$ and integrating over $[0,T]$ with respect to $s$, we get
        \begin{equation}
        \label{m3}
            \begin{aligned}
                & \int_0^T \int_s^T \int_K \frac{\abs*{Du}^{p-2}Du (x,t-s) \cdot [D\phi]_l(x,t) e^{-\frac{s}{h}}}{h} - \frac{u^q (x,t-s) \partial_t [\phi]_l (x,t) e^{-\frac{s}{h}} }{h} \, dx dt ds \\
                + & \int_0^T\int_K u^q(x,T-s)  [\phi]_l(x,T) \frac{e^{- \frac{s}{h}}}{h}\, dx ds= \int_0^T\int_K u^q(x,0) [\phi]_l(x,s) \frac{e^{- \frac{s}{h}}}{h} \, dx ds  \\
                \Longrightarrow & I_1 -I_2 + I_3 = I_4 \, .
            \end{aligned}
        \end{equation}

        \noindent On $I_1$ apply Fubini-Tonelli's theorem and perform the change of variables $t-s=\tau'$ to get
        \begin{equation}
            \label{m4}
            \begin{aligned}
               I_1= & \int_0^T \int_0^t \int_K \frac{\abs*{Du}^{p-2}Du (x,t-s) \cdot [D\phi]_l(x,t) e^{-\frac{s}{h}}}{h} \, dx ds dt \\
                = & \int_0^T \int_K [D\phi]_l(x,t) \int_0^t \frac{\abs*{Du}^{p-2} Du (x,\tau')e^{\frac{\tau'-t}{h}}}{h} \, d\tau' \, dx dt \\
                =& \int_0^T \int_K [D\phi]_l(x,t) [\abs*{Du}^{p-2} Du]_h(x,t) \, dx dt \, .
            \end{aligned}
        \end{equation}
        On $I_2$ we repeat the previous passages
        \begin{equation}
            \label{m5}
            \begin{aligned}
                I_2=& \int_0^T \int_0^t \int_K \frac{u^q (x,t-s) \partial_t [\phi]_l(x,t) e^{-\frac{s}{h}}}{h} \, dx ds dt \\
                = & \int_0^T \int_K \partial_t [\phi]_{l}(x,t) \frac{1}{h} \int_0^t \frac{u^q(x,t-s) e^{-\frac{s}{h}}}{h} \, dx ds dt \\
                = & \int_0^T \int_K \partial_t \phi_l(x,t) [u^q]_h(x,t) \,dx dt \, .
            \end{aligned}
        \end{equation}
        On $I_3$ just perform the change of variables $T-s = \tau'$
        \begin{equation}
            \label{m6}
            \begin{aligned}
                I_3 = & \int_K [\phi]_l(x,T) \frac{1}{h} \int_0^T u^q (x,\tau') e^{\frac{\tau'-T}{h}} \, ds dx \\
                =& \int_K [\phi]_l(x,T) [u^q]_h(x,T) \, dx \, .
            \end{aligned}
        \end{equation}
        Putting  \eqref{m4}, \eqref{m5} and \eqref{m6} in \eqref{m3} we arrive at
        \begin{equation*}
        \begin{aligned}
            & \int_0^T \int_K [\abs*{Du}^{p-2}Du]_h [D\phi]_l \, dx dt - \int_0^T \int_K \partial_t[\phi]_l [u^q]_h \, dx dt \\
            + &\int_K [\phi]_l(x,T) [u^q]_h(x,T) \, dx = \int_K u^q(x,0) \int_0^T \frac{[\phi]_l(x,s) e^{-\frac{s}{h}}}{h} \, ds \, dx \ .
            \end{aligned}
        \end{equation*}
        Now use the differentiation by parts on the second term of the left hand side to get
        \begin{equation*}
        \begin{aligned}
            &\int_0^T \int_K [\abs*{Du}^{p-2} Du]_h [D\phi]_l \, dx dt + \int_0^T \int_K [\phi]_l \partial_t[u^q]_h \, dx dt \\
           -&  \int_K [\phi]_l [u^q]_h (x,T) \, dx dt + \int_K [\phi]_l [u^q](x,0) \, dx dt \\
          + &  \int_K [\phi]_l(x,T) [u^q]_h (x,T) \, dx = \int_K u^q(x,0) \int_0^T \frac{[\phi]_l(x,s) e^{-\frac{s}{h}}}{h} \, ds \, dx \, 
            \end{aligned}
        \end{equation*}
       and since $[\phi]_l (x,0) =0$, 
        \begin{equation*}
        \begin{aligned}
            &\int_0^T \int_K [\abs*{Du}^{p-2} Du]_h [D\phi]_l \, dx dt + \int_0^T \int_K [\phi]_l \partial_t[u^q]_h \, dx dt \\
           =& \int_K u^q(x,0) \int_0^T \frac{[\phi]_l(x,s) e^{-\frac{s}{h}}}{h} \, ds \, dx \, .
            \end{aligned}
        \end{equation*}
        Finally, due to Proposition \ref{propmol} (ii) and letting $l \to 0^+$, we get 
        \begin{equation*}
            \int_0^T \int_K \abs*{ [\phi]_l - \phi} \, dx dt \leq \biggl( \int_0^T \int_K \abs*{[\phi]_l - \phi}^p \biggl)^{\frac{1}{p}} (T \abs*{K})^{\frac{p-1}{p}} \to 0 \, 
        \end{equation*}
        A similar reasoning can be performed to estimate $[D\phi]_l$, and we can conclude that 
        \begin{equation*}
            \begin{aligned}
            &\int_0^T \int_K [\abs*{Du}^{p-2} Du]_h D\phi \, dx dt + \int_0^T \int_K \phi \partial_t[u^q]_h \, dx dt \\
           =& \int_K u^q(x,0) \int_0^T \frac{\phi(x,s) e^{-\frac{s}{h}}}{h} \, ds \, dx \, .
            \end{aligned}
        \end{equation*}
    \end{proof}

    \vskip 0.5 cm

    \begin{riconoscimenti}
         F.M. Cassanello is a member of GNAMPA (Gruppo Nazionale per l’Analisi Matematica, la Probabilità e le loro Applicazioni) of INdAM (Istituto Nazionale di Alta Matematica ’Francesco Severi’), and is supported by the research project \textit{Regolarità ed esistenza per operatori anisotropi} (GNAMPA, CUP E5324001950001).
         E. Henriques is a member of CMAT (Centro de Matem\'atica da  Universidade do Minho; Polo CMAT-UTAD), and is supported by Portuguese Funds through FCT - Funda\c c\~ao para a Ci\^encia e a Tecnologia - within the Project UID/00013/2025.
    \end{riconoscimenti}

\vskip4pt

\noindent
\textbf{Data Availability.} There is no data associated with this maunscript.

\vskip4pt

\noindent
\textbf{Conflict of interests.} The authors state no conflict of interests.


\end{document}